\documentclass[10.5pt,a4paper]{article}
\usepackage{amsmath,amssymb,amsthm}
\DeclareSymbolFont{yhlargesymbols}{OMX}{yhex}{m}{n} \DeclareMathAccent{\yhwidehat}{\mathord}{yhlargesymbols}{"62}
\usepackage{ascmac}
\usepackage{color}
\usepackage[utf8]{inputenc}
\definecolor{refkey}{rgb}{0.9451,0.2706,0.4941}
\definecolor{labelkey}{rgb}{0.9451,0.2706,0.4941}
\theoremstyle{definition}
 \newtheorem{dfn}{Definition}[section]
 \newtheorem{remark}[dfn]{Remark}

\theoremstyle{plain}
 \newtheorem{thm}[dfn]{Theorem}
 \newtheorem{prop}[dfn]{Proposition}
 \newtheorem{lem}[dfn]{Lemma}
 \newtheorem{cor}[dfn]{Corollary}
\newtheorem{assump}[dfn]{Assumption}

\numberwithin{equation}{section}

\newcommand{\bn}{{\bold n}}

\newcommand{\bu}{{\bold u}}

\newcommand{\bv}{{\bold v}}
\newcommand{\bw}{{\bold w}}

\newcommand{\bff}{{\bold f}}
\newcommand{\bA}{{\bold A}}
\newcommand{\bB}{{\bold B}}

\newcommand{\bD}{{\bold D}}

\newcommand{\bF}{{\bold F}}
\newcommand{\bG}{{\bold G}}
\newcommand{\bH}{{\bold H}}
\newcommand{\bI}{{\bold I}}

\newcommand{\bK}{{\bold K}}

\newcommand{\bM}{{\bold M}}

\newcommand{\bS}{{\bold S}}
\newcommand{\bU}{{\bold U}}

\newcommand{\DV}{{\rm Div}\,}
\newcommand{\dv}{{\rm div}\,}

\newcommand{\BR}{{\Bbb R}}
\newcommand{\BC}{{\Bbb C}}

\newcommand{\BN}{{\Bbb N}}

\newcommand{\tth}{\widetilde{h}}
\newcommand{\tbu}{\widetilde{\bu}}
\newcommand{\trho}{\widetilde{\rho}}
\newcommand{\tgamma}{\widetilde{\gamma}}

\newcommand{\tbn}{\widetilde{\bn}}

\newcommand{\CA}{{\mathcal A}}
\newcommand{\CB}{{\mathcal B}}
\newcommand{\CC}{{\mathcal C}}
\newcommand{\CD}{{\mathcal D}}
\newcommand{\CE}{{\mathcal E}}
\newcommand{\CF}{{\mathcal F}}

\newcommand{\CL}{{\mathcal L}}

\newcommand{\CN}{{\mathcal N}}
\newcommand{\CG}{{\mathcal G}}
\newcommand{\CR}{{\mathcal R}}
\newcommand{\CS}{{\mathcal S}}
\newcommand{\CT}{{\mathcal T}}
\newcommand{\CH}{{\mathcal H}}

\newcommand{\CU}{{\mathcal U}}
\newcommand{\CV}{{\mathcal V}}

\newcommand{\CX}{{\mathcal X}}

\newcommand{\CZ}{{\mathcal Z}}
\newcommand{\R}{{\mathbb R}}
\newcommand{\C}{{\mathbb C}}
\newcommand{\N}{{\mathbb N}}

\newcommand{\fA}{{\frak A}}

\newcommand{\fB}{{\frak B}}
\newcommand{\fC}{{\frak C}}
\newcommand{\fX}{{\frak X}}

\newcommand{\bg}{{\bold g}}

\newcommand{\pd}{\partial}

\renewcommand{\Re}{\operatorname{Re}}

\newcommand{\vp}{\varphi}
\renewcommand{\labelenumi}{{\rm(\arabic{enumi})}}

\title{On the $\CR$-boundedness of solution operators for a compressible fluid model of Korteweg type in general domains}
\author{Sri MARYANI
\thanks{Department of Mathematics, Faculty of Mathematics and Natural Sciences,
Jenderal Soedirman University,
\endgraf
e-mail address: sri.maryani@unsoed.ac.id
\endgraf
Supported by International Research Collaboration (IRC)'s scheme 2024, LPPM UNSOED
}\enskip and \enskip
Miho MURATA
\thanks{Department of Mathematical and System Engineering,
Faculty of Engineering,
Shizuoka University, 
\endgraf
3-5-1 Johoku, Chuo-ku, Hamamatsu-shi, Shizuoka,
432-8561, Japan.
\endgraf
e-mail address: murata.miho@shizuoka.ac.jp
\endgraf
Partially supported by 
JSPS Grants-in-Aid for Early-Career Scientists 21K13819 and Grants-in-Aid for Scientific Research (B) 22H01134}
}
\date{}

\begin{document}
\maketitle

\begin{abstract}
In this paper, we consider a resolvent problem 
arising from the free boundary problem for the compressible fluid model 
of Korteweg type, which is called the {\it Navier-Stokes-Korteweg system}, with surface tension in general domains.
The Navier-Stokes-Korteweg system describes the liquid-vapor two-phase flow with non-zero thickness phase boundaries, which is often called the diffuse interface model.
Our purpose is to show the solution operator families of the resolvent problem are $\CR$-bounded, which gives us the generation of analytic semigroup and 
the maximal regularity 
in the $L_p$-in-time and $L_q$-in-space setting by applying the Weis
operator valued Fourier multiplier theorem \cite{W}. 
\end{abstract}

\section{Introduction}\label{sec:intro}
\subsection{Model}\label{subsec:model}
In this paper, we consider the compressible fluid flow, especially liquid-vapor two-phase flow, with capillarity effects
in a general domain $\Omega \subset \R^N$, $N \ge 2$, which we define as a uniform $C^3$-domain.
The fluid flow is described by the following system:  
\begin{equation}\label{nsk}
\left\{
\begin{aligned}
	&\pd_t \rho + \dv (\rho \bu) = 0 & &\quad\text{in $\Omega$, $t>0$}, \\
	&\rho (\pd_t \bu + \bu \cdot \nabla \bu)  
	= \DV (\bS(\bu)+\bK(\rho) - P(\rho)\bI)& &\quad\text{in $\Omega$, $t>0$}	
\end{aligned}
\right.
\end{equation}
with the initial condition and boundary conditions which we denote below,
where $\pd_t = \pd/\pd t$, $t$ is the time variable.
Here $\rho = \rho(x)$, $x=(x_1, \ldots, x_N)$
and  
$\bu = \bu(x) = (u_1(x), \ldots, u_N(x))^\mathsf{T} \footnote{$\bA^\mathsf{T}$denotes the transpose of $\bA$.}$
are unknown density field and velocity field, respectively;
$P(\rho)$ is the pressure field 
satisfying a $C^\infty$ function defined on
$\rho\ge 0$;
$\bI$ 
is the $N\times N$ identity matrix.
For any vector of functions $\bv = (v_1, \ldots, v_N)^\mathsf{T}$, 
$\dv \bv = \sum_{j=1}^N\pd_jv_j$ with $\pd_j=\pd/\pd x_j$,
and also for any $N\times N$ matrix field $\bM$, $\DV \bM = (\sum_{k=1}^N\pd_k M_{1k}, \dots, \sum_{k=1}^N\pd_k M_{Nk})^\mathsf{T}$.
The viscous stress tensor $\bS(\bu)$ is given by
\begin{align*}
\bS (\bu) &= \mu \bD(\bu) + 
(\nu - \mu) \dv \bu \bI, 
\end{align*} 
where $\mu$ and $\nu$ are the viscosity coefficients,
$\bD(\bu)$ denotes 
the deformation tensor whose $(j, k)$ components are
 $D_{jk}(\bu) = \pd_ju_k
+ \pd_ku_j$.
In addition,
$\bK(\rho)$ is the capillary stress tensor, which is called the Korteweg stress tensor, 
given by 
\begin{equation}\label{ks}
	\bK (\rho) = \frac{\kappa}{2} (\Delta \rho^2 -  |\nabla \rho|^2 )\bI 
- \kappa \nabla \rho \otimes \nabla \rho,
\end{equation}
 where 
 $\kappa$ is the capillary coefficient;
$\nabla \rho \otimes \nabla \rho$ denotes an $N\times N$ matrix with $(j, k)^{\rm th}$
component $(\pd_j\rho)(\pd_k\rho)$
for $\nabla \rho = (\pd_1 \rho, \dots, \pd_N \rho)^\mathsf{T}$.
The system \eqref{nsk} is often called the {\it Navier-Stokes-Korteweg system}.
Note that \eqref{nsk} without $\bK(\rho)$ is equivalent to the usual model of compressible viscous fluids.

The formulation of the theory of capillarity with diffuse interfaces was first introduced by Korteweg \cite{K}.
Korteweg proposed an explicit expression for the capillary stress tensor including density gradients to explain the non-local molecular interaction within the interface.
Afterward, Dunn and Serrin \cite{DS} introduced the capillary stress tensor \eqref{ks} by an approach of rational mechanics.

We recall that mathematical results on \eqref{nsk}.
The whole-space problem has been studied by many authors. 
Hattori and Li \cite{HL1, HL2} first showed
the local and global unique existence of the strong solution
in Sobolev space. 
Hou, Peng, and Zhu \cite{HPZ} improved the results \cite{HL1, HL2}
when the total energy is small. 
Danchin and Desjardins \cite{DD} proved the local and global existence and uniqueness of the solution 
in the critical Besov space.
The existence theorem in the case of zero sound speed $P'(\rho_*) = 0$ with some constant $\rho_*$,
we refer to \cite{CK, KT, KM, SX23}.
Large-time behavior of solutions was established by 
\cite{WT, TW, TWX, TZ} 
in the $L_2$-setting; \cite{MS} in the $L_p$-$L_q$ setting.
For the optimal decay estimates, we refer to \cite{CK} in the $L_2$ critical Besov spaces; \cite{KSX} in the $L_p$-$L_q$ framework;
\cite{SX} in the $L_p$ critical Besov spaces.
In the non-isentropic case, we refer to \cite{HL96, HPZ18}.
The existence of global weak solution was studied by \cite{BDL, H2011}.

Regarding boundary value problems,
the system \eqref{nsk} was studied in a domain $\Omega$ with the Dirichlet boundary condition:
\begin{equation}\label{bc}
	\bu=0, \quad \bn \cdot \nabla \rho=0 \quad\text{on $\Gamma$.}
\end{equation}
Kotschote \cite{K2008} proved the existence and uniqueness of strong solutions in both bounded and exterior domains locally in time in the $L_p$-setting.
He also considered a non-isothermal case for Newtonian and Non-Newtonian fluids in \cite{K2010, K2012}.
Moreover, he proved the global existence and exponential decay estimates of strong solutions in a bounded domain for small initial data in \cite{K2014}. 
Saito \cite{S2020} proved the existence of $\CR$-bounded solution operator families 
for the large resolvent parameter, and then he obtained the maximal $L_p$-$L_q$ regularity for the linearized problem in uniform $C^3$-domains and the generation of analytic semigroup
associated with the linearized problem.
Concerning the resolvent estimates with small resolvent parameters,
Kobayashi, the second author, and Saito \cite{KMS} proved the resolvent estimate for $\lambda \in \overline{\C_+}=\{z \in \C \mid \Re z\ge 0\}$ in a bounded domain
under the condition that the pressure $P(\rho)$ satisfies not only $P'(\rho_*) \ge 0$
but also $P'(\rho_*)<0$, and then \cite{KMS} obtained a global solvability.
Moreover, they proved the resolvent estimate for $\lambda \in \overline{\C_+}$ with $|\lambda| \ge \delta$ for any $\delta>0$ in an exterior domain if $P'(\rho_*) \ge 0$.

In \cite{IS}, the local solvability for the system \eqref{nsk} in general domains with the Dirichlet and the slip boundary conditions was proved in the $L_p$-$L_q$ setting. 

On the other hand, we will discuss a free boundary problem with the surface tension;
namely, 
the domain $\Omega$ and the boundary condition \eqref{bc} are replaced by a time-dependent domain $\Omega_t$ and
\begin{equation}\label{bc nsk}
\left\{
\begin{aligned}
	&(\bS(\bu)+\bK(\rho)-P(\rho)\bI)\bn_t
	=-P(\rho_*)\bn_t+\sigma(H(\Gamma_t)-H(\Gamma_0))\bn_t& &\quad\text{on $\Gamma_t$, $t>0$}, \\
	&\bn_t \cdot \nabla \rho=0 & &\quad\text{on $\Gamma_t$, $t>0$}, \\	
	&V_{\Gamma_t}=\bu \cdot \bn_t & &\quad\text{on $\Gamma_t$, $t>0$},
\end{aligned}
\right.
\end{equation}
respectively, where 
$\bn_t$ is the unit outer normal;
$\sigma$ is the coefficient of the surface tension;
$H(\Gamma_t)$ is the $N-1$-fold mean curvature on $\Gamma_t$;
$\Gamma_0$ is the boundary of the given initial domain $\Omega_0$.
The third equation of \eqref{bc nsk} is called the kinematic equation, where
$V_{\Gamma_t}$ is the velocity of the evolution of free surface $\Gamma_t$ in the direction of $\bn_t$.
 Since $\Omega_t$ is unknown, we transform $\Omega_t$ to the fixed domain $\Omega_0$ by the so-called {\it Hanzawa transform} \cite{H}, 
and then we obtain the linearized equations in $\Omega$ as follows:
\begin{equation}\label{linear}
\left\{
\begin{aligned}
	&\pd_t \rho + \gamma_1 \dv \bu = d & \quad&\text{in $\Omega$, $t>0$}, \\
	&\pd_t \bu - \gamma_4^{-1} 
	\DV\{\gamma_2 \bD(\bu) + (\gamma_3 - \gamma_2) \dv \bu \bI + \gamma_1  \Delta \rho \bI\}= \bff& \quad&\text{in $\Omega$, $t>0$},\\
	&\{\gamma_2 \bD(\bu) + (\gamma_3 - \gamma_2) \dv \bu \bI + \gamma_1  \Delta \rho \bI \} \bn - \sigma \Delta_{\Gamma} h \bn =\bg & \quad&\text{on $\Gamma$, $t>0$},\\
	&\bn \cdot \nabla \rho = k & \quad&\text{on $\Gamma$, $t>0$},\\
	&\pd_t h - \bu \cdot \bn = \zeta & \quad&\text{on $\Gamma$, $t>0$}\\
	&(\rho, \bu, h)|_{t=0} = (\rho_0, \bu_0, h_0) & \quad&\text{in $\Omega$}.
\end{aligned}
\right.
\end{equation}
Applying the Laplace transform to \eqref{linear}, we obtain the resolvent problem: 
\begin{equation}\label{main prob}
\left\{
\begin{aligned}
	&\lambda \rho + \gamma_1 \dv \bu = d & \quad&\text{in $\Omega$}, \\
	&\lambda \bu - \gamma_4^{-1} 
	\DV\{\gamma_2 \bD(\bu) + (\gamma_3 - \gamma_2) \dv \bu \bI + \gamma_1  \Delta \rho \bI\}= \bff& \quad&\text{in $\Omega$},\\
	&\{\gamma_2 \bD(\bu) + (\gamma_3 - \gamma_2) \dv \bu \bI + \gamma_1  \Delta \rho \bI \} \bn - \sigma \Delta_{\Gamma} h \bn =\bg & \quad&\text{on $\Gamma$},\\
	&\bn \cdot \nabla \rho = k & \quad&\text{on $\Gamma$},\\
	&\lambda h - \bu \cdot \bn = \zeta & \quad&\text{on $\Gamma$}.
\end{aligned}
\right.
\end{equation}
Concerning a free boundary problem, Saito \cite{S} considered the resolvent problem arising from a free boundary problem without surface tension in $\R^N_+$;
namely, $\sigma=0$ in \eqref{bc nsk}.
He constructed $\CR$-bounded solution operator families for the resolvent problems in $\R^N$ and $\R^N_+$. 
In \cite{MM}, we recently studied the resolvent problem into account surface tension in $\R^N_+$ by the same motivation as \cite{S}.

As a follow-up study of \cite{MM}, we discuss the existence of the $\CR$-bounded solution operator families for the resolvent problem \eqref{main prob} in a general domain. 
The $\CR$-boundedness for the solution operator families implies that the generation of analytic semigroup and the maximal $L_p$-$L_q$ regularity for the linearized equations, which is the key issue when we consider the local solvability for the nonlinear problem in the maximal $L_p$-$L_q$ regularity class.

This paper is organized as follows.
In the next section, we state the main theorem concerning the $\CR$-bounded operator families for the Navier-Stokes-Korteweg system with surface tension in a uniform $C^3$-domain. 
As an application of the main theorem, the generation of analytic semigroup and the maximal $L_p$-$L_q$ regularity for the linearized equations are also obtained.
In Section \ref{sec.w}, we recall the results for the whole-space problem and the half-space problem.
Section \ref{sec.bent} proves the $\CR$-boundedness for the solution operators in the bent half-space. 
Section \ref{sec.general1} treats the resolvent problem in a uniform $C^3$-domain.
In Section \ref{sec.general1}, the problem is reduced to the problem in the whole-space and the bent half-space by the localization technique. Then we prove the remainder terms are small in the suitable norms as in Lemma \ref{lem:Rbound V g}. To prove Lemma \ref{lem:Rbound V g}, we need the estimate for $\lambda^{3/2} \rho$.
Because of this reason, we first construct the solutions for the right member including the term $\lambda^{1/2}d$ 
in Section \ref{sec.w}, Section \ref{sec.bent}, and Section \ref{sec.general1}. 
Then we finally construct the solutions for the right member without $\lambda^{1/2}d$ in Section \ref{sec:general2}.

\subsection{Notation and Assumptions}

We summarize several symbols and functional spaces used 
throughout the paper.
Let $\N$, $\R$, and $\C$ denote the sets of 
all natural numbers, real numbers, and complex numbers, respectively. 
Set $\N_0=\N \cup \{0\}$ and $\C_{+, \omega} = \{z \in \C \mid \Re z > \omega\}$ for $\omega \ge 1$.

We use boldface letters, e.g. $\bu$ to 
denote vector-valued functions. 

For scalar function $f$ and $N$-vector functions $\bg$, we set
\begin{align*}
\nabla f &= (\pd_1f,\ldots,\pd_Nf)^\mathsf{T},&
\enskip 
\nabla^2 f &= (\pd_i\pd_j f)_{1\le i, j\le N},\\
\nabla^3 f &= \{\pd_i \pd_j \pd_k f \mid i, j, k = 1,\ldots, N \},\\ 
\nabla \bg &= (\pd_j g_i)_{1\le i, j\le N},&
\enskip
\nabla^2 \bg &= \{\pd_i \pd_j g_k \mid i, j, k = 1,\ldots, N \},&
\end{align*} 
where $\pd_i = \pd/\pd x_i$.

Let $\Omega$ be a domain of $\R^N$.
For $N$-vector functions $\bff=(f_1(x), \ldots, f_N(x))^\mathsf{T}$ and $\bg=(g_1(x), \ldots, g_N(x))^\mathsf{T}$ on $\Omega$;
matrix-valued functions $\bA=(a_{ij}(x))_{1\le i, j\le N}$ and $\bB=(b_{ij}(x))_{1\le i, j\le N}$ on $\Omega$, 
set
\begin{align*}
(\bff, \bg)_{\Omega}&= \sum^N_{j=1} \int_{\Omega} f_j(x) g_j(x)\,dx, &
(\bff, \bg)_{\Gamma}&= \sum^N_{j=1} \int_{\Gamma} f_j(x) g_j(x)\,dS, \\
(\bA, \bB)_{\Omega}&= \sum^N_{i, j=1} \int_{\Omega} a_{ij}(x) b_{ij}(x)\,dx, &
(\bA, \bB)_{\Gamma}&= \sum^N_{i, j=1} \int_{\Gamma} a_{ij}(x) b_{ij}(x)\,dS,
\end{align*}
where $dS$ denotes the surface element of $\Gamma$.

For Banach spaces $X$ and $Y$, $\CL(X,Y)$ denotes the set of 
all bounded linear operators from $X$ into $Y$,
$\CL(X)$ is the abbreviation of $\CL(X, X)$, and 
$\rm{Hol}\,(U, \CL(X,Y))$ denotes
 the set of all $\CL(X,Y)$ valued holomorphic 
functions defined on a domain $U$ in $\BC$. 

For $1 \le p \le \infty$ and $m \in \N$,
$L_p(D)$ and $H_p^m(D)$ 
denote the Lebesgue space and Sobolev space on domain $D \subset \R^N$;
while $\|\cdot\|_{L_q(D)}$ and $\|\cdot\|_{H_q^m(D)}$
denote their norms, respectively.
In addition, $B^s_{q, p}(D)$ is the Besov space on domain $D$ for $1 < q < \infty$ and $s>0$ with the norm $\|\cdot\|_{B_{q, p}^s(D)}$.
For simplicity, we write $W^s_q(D) = B^s_{q, q}(D)$.
The $d$-product space of $X$ is defined by 
$X^d=\{f=(f, \ldots, f_d) \mid f_i \in X \, (i=1,\ldots,d)\}$,
while its norm is denoted by 
$\|\cdot\|_X$ instead of $\|\cdot\|_{X^d}$ for the sake of 
simplicity. 

For the Banach space X, we also denote the usual Lebesgue space and Sobolev space of $X$-valued functions 
defined on time interval $I$ by $L_p(I, X)$ and $H^m_p(I, X)$ with $m \in \N$; 
while $\|\cdot\|_{L_p(I, X)}$, $\|\cdot\|_{H_p^m(I, X)}$
denote their norms, respectively.

Let $\gamma \ge 1$.
Set
\begin{align*}
L_{p, \gamma}(\R, X) & = \{f(t) \in L_{p, {\rm loc}}
(\R, X) \mid e^{-\gamma t}f(t) \in L_p(\R, X)\}, \\
L_{p, \gamma, 0}(\R, X) & = \{f(t) \in L_{p, \gamma}(\R, X) 
\mid
f(t) = 0 \enskip (t < 0)\}, \\
H^m_{p, \gamma}(\R, X) & = \{f(t) \in
 L_{p, \gamma}(\R, X) \mid 
e^{-\gamma t}\pd_t^j f(t) \in L_p(I, X)
\enskip (j=1, \ldots, m)\}, \\
H^m_{p, \gamma, 0}(\R, X) & = H^m_{p, \gamma}(\R, X)
\cap L_{p, \gamma, 0}(\R, X).
\end{align*}
Let $\CL$ and $\CL^{-1}$ denote the Laplace transform and the
Laplace inverse transform, respectively, which are defined by 
\[
\CL[f](\lambda) = \int_{\R} e^{-\lambda t} f(t) \,dt, \quad
\CL^{-1}[g](t) = \frac{1}{2\pi}
\int_{\R} e^{\lambda t} g(\tau) \,d\tau
\]
with $\lambda = \gamma + i\tau \in \C$.  Given $s \ge 0$
and $X$-valued function $f(t)$, we set 
\[
(\Lambda^s_\gamma f)(t) = \CL^{-1} [\lambda^s
\CL[f](\lambda)](t).
\]
The Bessel potential space of $X$-valued functions 
of order $s$ is defined by the following:
\begin{align*}
H^s_{p, \gamma}(\R, X) & = \{f \in L_p(\R, X) \mid 
e^{-\gamma t} (\Lambda_\gamma^s f)(t) \in L_p(\R, X)\}, \\
H^s_{p, \gamma, 0}(\R, X) & = \{f \in H^s_{p, \gamma}
(\R, X) \mid f(t) = 0 \enskip(t < 0)\}.
\end{align*}

The letter $C$ denotes generic constants and the constant 
$C_{a,b,\ldots}$ depends on $a,b,\ldots$. 
The values of constants $C$ and $C_{a,b,\ldots}$ 
may change from line to line. 

Here we introduce the definition of uniform $C^3$-domains and assumptions for the coefficients $\gamma_j$ ($j=1, 2, 3, 4$).
\begin{dfn}\label{def:general domain}
Let $\Omega$ be a domain of $\R^N$ with the boundary $\Gamma \neq \emptyset$.
Then $\Omega$ is called a uniform $C^3$-domain, if there exist positive constants $a_1$, $a_2$, and $K$ such that the following assertions hold:
For any $x_0 = (x_{01}, \dots, x_{0N}) \in \Gamma$, there exist a coordinate number $j$ and a $C^3$ function $h(x')$ defined on $B'_{a_1}(x_0')$ such that $\|h\|_{H^3_\infty(B_{a_1}'(x_0'))} \le K$ and
\begin{align*}
	\Omega \cap B_{a_2}(x_0) &= \{x \in \R^N \mid x_j > h(x')~(x' \in B_{a_1}'(x'_0))\} \cap B_{a_2}(x_0),\\
	\Gamma \cap B_{a_2}(x_0) &= \{x \in \R^N \mid x_j = h(x')~(x' \in B_{a_1}'(x'_0))\} \cap B_{a_2}(x_0).
\end{align*}
Here we have set
\begin{align*}
	x' &= (x_{1}, \dots, x_{j-1}, x_{j+1}, \dots, x_{N}),\\
	x'_0 &= (x_{01}, \dots, x_{0j-1}, x_{0j+1}, \dots, x_{0N}),\\
	B'_{a_1}(x_0')&=\{x' \in \R^{N-1} \mid |x'-x'_0| < a_1\},\\
	B_{a_2}(x_0)&=\{x \in \R^N \mid |x-x_0| < a_2\}.
\end{align*}
\end{dfn}

\begin{assump}\label{assumption gamma}
The coefficient $\gamma_j = \gamma_j(x)$ $(j=1, 2, 3, 4)$ satisfy the following conditions.

\begin{enumerate}
\item
The coefficient $\gamma_j = \gamma_j(x)$ are real-valued uniformly Lipschitz continuous functions on $\overline \Omega$:
there exists a positive constant $L$ such that $|\gamma_j(x) - \gamma_j(y)| \le L |x-y|$ for any $x, y \in \overline \Omega$ and $j = 1, 2, 3, 4$.

\item
There exist positive constants $\gamma_*$, $\gamma^*$ such that $\gamma_* \le \gamma_j(x) \le \gamma^*$ for any $x \in \overline \Omega$ and $j = 1, 2, 3, 4$.

\item
The conditions
$(\gamma_2(x) + \gamma_3(x))^2 \neq 2(\gamma_1(x))^2\gamma_4(x)$ and $(\gamma_1(x))^2 \gamma_4(x) \neq \gamma_2(x) \gamma_3(x)$
hold for any $x \in \Gamma$.
\end{enumerate}
\end{assump}
\begin{remark}\label{rem:lip}
Let $C^{0, 1}(\overline \Omega)$ be a Banach space of all bounded and uniformly Lipschitz continuous functions on $\overline \Omega$ with the norm:
\[
	\|f\|_{C^{0, 1}(\overline \Omega)} = \|f\|_{L_\infty(\Omega)} + \sup_{x, y \in \Omega, x \neq y} \frac{|f(x)-f(y)|}{|x-y|}.
\]
Note that $C^{0, 1}(\overline \Omega) \subset H^1_\infty(\Omega)$ with $\|f\|_{H^1_\infty(\Omega)} \le \|f\|_{C^{0, 1}(\overline \Omega)}$ for any $f \in C^{0, 1}(\overline \Omega)$
(cf.\,e.g.\,\cite[Lemma 3.12]{T}).
\end{remark}

\section{Main theorem}
To state the main theorem, we introduce
the definition of the $\CR$-boundedness of operator families and notation.

\begin{dfn}\label{dfn2}
Let $X$ and $Y$ be Banach spaces, and let $\CL(X,Y)$ be the set of 
all bounded linear operators from $X$ into $Y$.
A family of operators $\CT \subset \CL(X,Y)$ is called $\CR$-bounded 
on $\CL(X,Y)$, if there exist constants $C > 0$ and $p \in [1,\infty)$ 
such that for any $m \in \BN$, $\{T_{j}\}_{j=1}^{m} \subset \CT$,
$\{f_{j}\}_{j=1}^{m} \subset X$ and sequences $\{r_{j}\}_{j=1}^{m}$
 of independent, symmetric, $\{-1,1\}$-valued random variables on $[0,1]$, 
we have  the inequality:
$$
\bigg \{ \int_{0}^{1} \|\sum_{j=1}^{m} r_{j}(u)T_{j}f_{j}\|_{Y}^{p}\,du
 \bigg \}^{1/p} \leq C\bigg\{\int^1_0
\|\sum_{j=1}^m r_j(u)f_j\|_X^p\,du\biggr\}^{1/p}.
$$ 
The smallest such $C$ is called $\CR$-bound of $\CT$, 
which is denoted by $\CR_{\CL(X,Y)}(\CT)$.
\end{dfn}
\begin{remark}\label{rem:def of rbdd}
The $\CR$-boundedness implies that the uniform boundedness of the operator family $\CT$.
In fact, choosing $m=1$ in Definition \ref{dfn2}, we observed that there exists a constant $C$ such that $\|T f\|_Y \le C \|f\|_X$ holds for any $T \in \CT$ and $f \in X$.
\end{remark}

Set 
\begin{equation}\label{def:space}
\begin{aligned}
	X_q(D) &=  H^1_q(D) \times L_q(D)^N \times H^1_q(D) ^N
	\times H^2_q(D) \times W^{2-1/q}_q(\pd D),\\
	\CX_q(D) &=H^1_q(D) \times L_q(D)^\CN \times W^{2-1/q}_q(\pd D),\quad
	\CN = N+N^2+N+N^2+N+1,\\
	X^0_q(D) &=  H^1_q(D) \times L_q(D)^N \times H^1_q(D) ^N
	\times H^2_q(D) \times H^2_q(D),\\
	\CX^0_q(D) &=H^1_q(D) \times L_q(D)^\CN \times H^2_q(D),\\
	\CF_\lambda \bF &=(d, \bff, \nabla \bg, \lambda^{1/2} \bg, \nabla^2 k, \nabla \lambda^{1/2}k, \lambda k, \zeta),\\
	\fA_q (D) &= L_q(D)^{N^3+N^2} \times H^1_q(D), \quad \CR_\lambda \rho =(\nabla^3 \rho, \lambda^{1/2}\nabla^2\rho, \lambda \rho),\\
	\fB_q(D) &= L_q(D)^{N^3+N^2+N},
	\quad \CS_\lambda \bu =(\nabla^2 \bu, \lambda^{1/2}\nabla \bu, \lambda \bu),\\
	\fC_q(\pd D) &= W^{3-1/q}_q(\pd D) \times W^{2-1/q}_q(\pd D), \quad \fC_q(D) = H^3_q(D) \times H^2_q(D),
	\quad \CT_\lambda h =(h, \lambda h),
\end{aligned}
\end{equation}
where $D$ is a domain of $\R^N$ with the boundary $\pd D$.

Let us state the main result in this paper.
\begin{thm}\label{thm:main Rbound'}
Let $1<q<\infty$, and let $\Omega$ be a uniform $C^3$-domain. 
Assume that $\gamma_j$ $(j =1, 2, 3, 4)$ satisfies Assumption \ref{assumption gamma}.
Then there exists a constant $\lambda_0 \ge 1$ such that the following assertions hold true:

$\thetag1$ 
For any $\lambda \in \C_{+, \lambda_0}$ there exist operators  
\begin{align*}
&\CA (\lambda) \in 
{\rm Hol} (\C_{+, \lambda_0}, 
\CL(\CX_q(\Omega), H^3_q(\Omega)))\\
&\CB (\lambda) \in 
{\rm Hol} (\C_{+, \lambda_0}, 
\CL(\CX_q(\Omega), H^2_q(\Omega)^N)),\\
&\CC (\lambda) \in 
{\rm Hol} (\C_{+, \lambda_0}, 
\CL(\CX_q(\Omega), W^{3-1/q}_q(\Gamma)))
\end{align*}
such that 
for any $\bF=(d, \bff, \bg, k, \zeta) \in X_q(\Omega)$, 
\begin{equation*}
\rho = \CA (\lambda) \CF_\lambda \bF, \quad
\bu = \CB (\lambda) \CF_\lambda \bF, \quad
h = \CC (\lambda) \CF_\lambda \bF
\end{equation*}
are unique solutions of problem \eqref{main prob}.

$\thetag2$ 
There exists a positive constant $r$ such that
\begin{equation} \label{rbdd general}
\begin{aligned}
&\CR_{\CL(\CX_q(\Omega), \fA_q (\Omega))}
(\{(\tau \pd_\tau)^n \CR_\lambda \CA (\lambda) \mid 
\lambda \in \C_{+, \lambda_0}\}) 
\leq r,\\
&\CR_{\CL(\CX_q(\Omega), \fB_q(\Omega))}
(\{(\tau \pd_\tau)^n \CS_\lambda \CB (\lambda) \mid 
\lambda \in \C_{+, \lambda_0}\}) 
\leq r,\\
&\CR_{\CL(\CX_q(\Omega), \fC_q(\Gamma))}
(\{(\tau \pd_\tau)^n \CT_\lambda \CC (\lambda) \mid 
\lambda \in \C_{+, \lambda_0}\}) 
\leq r
\end{aligned}
\end{equation}
for $n = 0, 1$.
Here, above constant $r$ depend solely on $N$, $q$, $L$, $\gamma_*$ and $\gamma^*$. 
\end{thm}

Since there exist linear mappings 
\[
	\CT : H^n_q(\Omega) \to W^{n-1/q}_q(\Gamma),
\quad
	\CE : W^{n-1/q}_q(\Gamma) \to H^n_q(\Omega)
\]
such that $\|\CT f\|_{W^{n-1/q}_q(\Gamma)} \le C\|f\|_{W^n_q(\Omega)}$
and $\|\CE g\|_{W^n_q(\Omega)}\le C\|g\|_{W^{n-1/q}_q(\Gamma)}$ for $n=2, 3$ (cf. \cite[Proposition 9.5.4]{S2016}),
Theorem \ref{thm:main Rbound'} follows from the following theorem. 

\begin{thm}\label{thm:main Rbound}
Let $1<q<\infty$, and let $\Omega$ be a uniform $C^3$-domain. 
Assume that $\gamma_j$ $(j =1, 2, 3, 4)$ satisfies Assumption \ref{assumption gamma}.
Then there exists a constant $\lambda_0 \ge 1$ such that the following assertions hold true:

$\thetag1$ 
For any $\lambda \in \C_{+, \lambda_0}$ there exist operators  
\begin{align*}
&\CA (\lambda) \in 
{\rm Hol} (\C_{+, \lambda_0}, 
\CL(\CX^0_q(\Omega), H^3_q(\Omega)))\\
&\CB (\lambda) \in 
{\rm Hol} (\C_{+, \lambda_0}, 
\CL(\CX^0_q(\Omega), H^2_q(\Omega)^N)),\\
&\CC (\lambda) \in 
{\rm Hol} (\C_{+, \lambda_0}, 
\CL(\CX^0_q(\Omega), H^3_q(\Omega)))
\end{align*}
such that 
for any $\bF=(d, \bff, \bg, k, \zeta) \in X_q^0(\Omega)$, 
\begin{equation*}
\rho = \CA (\lambda) \CF_\lambda \bF, \quad
\bu = \CB (\lambda) \CF_\lambda \bF, \quad
h = \CC (\lambda) \CF_\lambda \bF
\end{equation*}
are unique solutions of problem \eqref{main prob}.

$\thetag2$
There exists a positive constant $r$ such that
\begin{equation} \label{rbdd general2}
\begin{aligned}
&\CR_{\CL(\CX^0_q(\Omega), \fA_q (\Omega))}
(\{(\tau \pd_\tau)^n \CR_\lambda \CA (\lambda) \mid 
\lambda \in \C_{+, \lambda_0}\}) 
\leq r,\\
&\CR_{\CL(\CX^0_q(\Omega), \fB_q(\Omega))}
(\{(\tau \pd_\tau)^n \CS_\lambda \CB (\lambda) \mid 
\lambda \in \C_{+, \lambda_0}\}) 
\leq r,\\
&\CR_{\CL(\CX^0_q(\Omega), \fC_q(\Omega))}
(\{(\tau \pd_\tau)^n \CT_\lambda \CC (\lambda) \mid 
\lambda \in \C_{+, \lambda_0}\}) 
\leq r
\end{aligned}
\end{equation}
for $n = 0, 1$.
Here, above constant $r$ depend solely on $N$, $q$, $L$, $\gamma_*$ and $\gamma^*$. 
\end{thm} 
The method to prove the existence of $\CR$-bounded solution operator families stated in Theorem \ref{thm:main Rbound'} by Theorem \ref{thm:main Rbound} is the same as the proof of \cite[Theorem 1]{MM} and the uniqueness of the solution to \eqref{main prob} in Theorem \ref{thm:main Rbound'} follows from Theorem \ref{thm:main Rbound}, thus we omit the proof of Theorem \ref{thm:main Rbound'}.

Thanks to Theorem \ref{thm:main Rbound'}, we can obtain the results for the generation of analytic semigroup and the maximal $L_p$-$L_q$ maximal regularity.
First, we state the generation of analytic semigroup.
Let
\[
	\fX_q(\Omega) = H^1_q(\Omega) \times L_q(\Omega)^N \times W^{2-1/q}_q(\Gamma)
\]
with the norm
\[
	\|(\rho, \bu, h)\|_{\fX_q(\Omega)}=\|\rho\|_{H^1_q(\Omega)} + \|\bu\|_{L_q(\Omega)} + \|h\|_{W^{2-1/q}_q(\Gamma)}
\]
and let $\CA$, $\CD(\CA)$, and $\|\cdot\|_{\CD(\CA)}$ be an operator, its domain, and the norm with
\begin{align*}
	\CD_q(\CA)&=\{(\rho, \bu, h) \in H^3_q(\Omega) \times H^2_q(\Omega)^N \times W^{3-1/q}_q(\Gamma) \\
	&\qquad \mid \{\gamma_2 \bD(\bu) + (\gamma_3 - \gamma_2) \dv \bu \bI + \gamma_1  \Delta \rho \bI \} \bn - \sigma \Delta_{\Gamma} h \bn=0, \enskip
	\bn \cdot \nabla \rho = 0 \enskip \text{on $\Gamma$} \},\\
	\|(\rho, \bu, h)\|_{\CD_q(\CA)}&=\|\rho\|_{H^3_q(\Omega)} + \|\bu\|_{H^2_q(\Omega)} + \|h\|_{W^{3-1/q}_q(\Gamma)},\\
	\CA(\rho, \bu, h)&=(-\gamma_1 \dv \bu, \enskip \gamma_4^{-1} \DV\{\gamma_2 \bD(\bu) + (\gamma_3 - \gamma_2) \dv \bu \bI + \gamma_1  \Delta \rho \bI\}, \enskip
	\bu \cdot \bn) \enskip \text{for $(\rho, \bu, h) \in \CD_q(\CA)$}. 
\end{align*}
Theorem \ref{thm:main Rbound'}, together with Remark \ref{rem:def of rbdd}, implies that there exist positive constant 
$\widetilde \lambda_0 > \lambda_0$ and $\epsilon_0 \in (0, \pi/2)$ such that for any $\lambda \in \Sigma_{\epsilon, \widetilde \lambda_0}$ ($\epsilon \in (\epsilon_0, \pi/2)$)
and $(d, \bff, 0, 0, \zeta) \in X_q(\Omega)$, \eqref{main prob} with $(\bg, k)=(0, 0)$ has a unique solution $ (\rho, \bu, h) \in \CD_q(\CA)$ satisfying
\begin{equation}\label{resolvent est}
	|\lambda|\| (\rho, \bu, h) \|_{\fX_q(\Omega)} + \| (\rho, \bu, h) \|_{\CD_q(\CA)} \le C\|(d, \bff, \zeta)\|_{\fX_q(\Omega)} 
\end{equation}
with some positive constant $C$ depending on $N, q, \widetilde \lambda_0$ (cf. \cite[Proposition 2.1.11]{L}).
Here
\[
	\Sigma_{\epsilon, \widetilde \lambda_0} = \{\widetilde \lambda_0 + z \mid z \in \Sigma_\epsilon\},
	\quad
	\Sigma_{\epsilon} = \{z \in \C \setminus\{0\} \mid |\arg z| < \pi - \epsilon\}.
\]
Then the following theorem follows from \eqref{resolvent est} and the standard arguments for the analytic semigroup.

\begin{thm}\label{thm:semi1}
Let $1 < q < \infty$.
Assume that Assumption \ref{assumption gamma} holds.
Then the operator $\CA$ generates an analytic
semigroup $\{T(t)\}_{t\geq 0}$ on $\fX_q(\Omega)$.  
Moreover, there exist
constants $\eta_0 \ge 1$ and $C_{q, N, \eta_0} > 0$
such that $\{T(t)\}_{t\geq 0}$ satisfies the estimates: 
\begin{align*}
	\|T(t) (\rho_0, \bu_0, h_0) \|_{\fX_q(\Omega)}
	&\leq C_{q, N, \eta_0} e^{\eta_0 t} \|(\rho_0, \bu_0, h_0)\|_{\fX_q(\Omega)},\\
	\|\pd_t T(t) (\rho_0, \bu_0, h_0) \|_{\fX_q(\Omega)}
	&\leq C_{q, N, \eta_0} e^{\eta_0 t} t^{-1} \|(\rho_0, \bu_0, h_0)\|_{\fX_q(\Omega)},\\
	\|\pd_t T(t) (\rho_0, \bu_0, h_0) \|_{\fX_q(\Omega)}
	&\leq C_{q, N, \eta_0} e^{\eta_0 t} \|(\rho_0, \bu_0, h_0)\|_{\CD_q(\CA)}
\end{align*}
for any $t > 0$.
\end{thm}
Let 
\[
	\CD_{q, p}(\R^N_+) = (\fX_q(\Omega), \CD_q(\CA))_{1-1/p, p}
\]
with $\|(\rho, \bu, h)\|_{\CD_{q, p} (\Omega)}=\|\rho\|_{B^{3-2/p}_{q, p}(\Omega)}+\|\bu\|_{B^{2(1-1/p)}_{q, p}(\Omega)}+\|h\|_{B^{3-1/q-1/p}_{q, p}(\Gamma)}$.
Combining Theorem \ref{thm:semi1} with a real interpolation method
(cf. Shibata and Shimizu \cite[Proof of Theorem 3.9]{SS}), we have
the following theorem. 

\begin{thm}\label{thm:semi2}
Let $1 < p, q < \infty$, and let $\eta_0$ be a positive constant stated in Theorem \ref{thm:semi1}. 
Assume that Assumption \ref{assumption gamma} holds.
Then for any $(\rho_0, \bu_0, h_0) \in \CD_{q, p} (\Omega)$, 
\eqref{linear} with $(d, \bff, \bg, k, \zeta) = (0, 0, 0, 0, 0)$
admits a unique solution $(\rho, \bu, h) = T(t) (\rho_0, \bu_0, h_0)$
possessing the estimate:
\begin{equation}\label{semi2}
\begin{aligned}
	&\|e^{-\eta t} \pd_t (\rho, \bu, h)\|_{L_p(\R_+, \fX_q(\Omega))}
	+ \|e^{-\eta t} (\rho, \bu, h)\|_{L_p(\R_+, \CD_q(\CA))}\\
	&\le C_{p, q, N, \eta_0} 
	\|(\rho_0, \bu_0, h_0)\|_{\CD_{q, p} (\Omega)}
\end{aligned}
\end{equation}
for any $\eta \ge \eta_0$. 
\end{thm}

Next, we state the maximal $L_p$-$L_q$ regularity.
Let
\begin{align*}
	\CF_{p, q, \eta}
	=\{(d, \bff, \bg, k, \zeta) \mid &d \in L_{p, \eta, 0}(\R, H^1_q(\Omega)), \enskip
	\bff \in L_{p, \eta, 0}(\R, L_q(\Omega)^N),\\
	&\bg \in L_{p, \eta, 0}(\R, H^1_q(\Omega)^N) \cap H^{1/2}_{p, \eta, 0}(\R, L_q(\Omega)^N), \\
	&k \in H^1_{p, \eta, 0}(\R, L_q(\Omega)) \cap L_{p, \eta, 0}(\R, H^2_q(\Omega)), \\
	&\zeta \in L_{p, \eta, 0}(\R, W^{2-1/q}_q(\Gamma))\}
\end{align*}
with the norm
\begin{align*}
	&\|(d, \bff, \bg, k, \zeta)\|_{\CF_{p, q, \eta}}\\
	&\enskip=\|e^{-\eta t} d \|_{L_p(\R, H^1_q(\Omega))} 
	+ \|e^{-\eta t} (\bff, \Lambda^{1/2}_\eta \bg) \|_{L_p(\R, L_q(\Omega))}
	+ \|e^{-\eta t} \bg \|_{L_p(\R, H^1_q(\Omega))}\\
	&\enskip+ \|e^{-\eta t} (\pd_t k, \Lambda^{1/2}\nabla k)\|_{L_p(\R, L_q(\Omega))}
	+ \|e^{-\eta t} k \|_{L_p(\R, H^2_q(\Omega))}
	+ \|e^{-\eta t} \zeta \|_{L_p(\R, W^{2-1/q}_q(\Gamma))}.
\end{align*}
In the same manner as in \cite[Theorem 9]{MM}, Theorem \ref{thm:main Rbound'}, together with the Weis's
operator valued Fourier multiplier theorem \cite{W}, furnishes the following theorem.

\begin{thm}\label{thm:mr} 
Let $1 < p, q < \infty$.
Assume that Assumption \ref{assumption gamma}
hold. 
Then
there exists a constant $\eta_1 \ge 1$ such that for any $(d, \bff, \bg, k, \zeta) \in \CF_{p, q, \eta}$,
\eqref{linear} with $(\rho_0, \bu_0, h_0)=(0, 0, 0)$ admits a unique solution 
$(\rho, \bu, h)$ with 
\begin{align*}
	&\rho \in L_{p, \eta_1, 0}(\R, H^3_q(\Omega)) \cap H^1_{p, \eta_1, 0}(\R, H^1_q(\Omega)), \\
	&\bu \in L_{p, \eta_1, 0}(\R, H^2_q(\Omega)^N) \cap H^1_{p, \eta_1, 0}(\R, L_q(\Omega)^N),\\
	&h \in L_{p, \eta_1, 0}(\R, W^{3-1/q}_q(\Gamma)) \cap H^1_{p, \eta_1, 0}(\R, W^{2-1/q}_q(\Gamma)),
\end{align*}
possessing the estimate
\begin{equation*}\label{est:mr}
\begin{aligned}
	&\|e^{-\eta t} \pd_t \rho\|_{L_p(\R, H^1_q(\Omega))} 
	+ \sum_{j=0}^3 \|e^{-\eta t} \Lambda^{j/2} \rho\|_{L_p(\R, H^{3-j}_q(\Omega))}\\
	&+ \|e^{-\eta t} \pd_t \bu\|_{L_p(\R, L_q(\Omega))}
	+ \sum_{j=0}^2 \|e^{-\eta t} \Lambda^{j/2} \bu\|_{L_p(\R, H^{2-j}_q(\Omega))}\\
	&+ \|e^{-\eta t} \pd_t h\|_{L_p(\R, W^{2-1/q}_q(\Gamma))}
	+ \|e^{-\eta t} h\|_{L_p(\R, W^{3-1/q}_q(\Gamma))}\\
	&\leq C\|(d, \bff, \bg, k, \zeta)\|_{\CF_{p, q, \eta}}
\end{aligned}
\end{equation*}
for any $\eta \ge \eta_1$ with some constant $C$ depending 
on $N$, $p$, $q$, and $\eta_1$. 
\end{thm}

As we mention above, Theorem \ref{thm:main Rbound'} follows from Theorem \ref{thm:main Rbound}, therefore we prove Theorem \ref{thm:main Rbound} below.

\section{Problems in the whole-space and the half-space}\label{sec.w}
In this section, we summarize the known results concerning $\CR$-boundedness for \eqref{main prob} in the whole-space and the half-space.
Set
\begin{align*}
	X^1_q(D) &=H^1_q(D) \times L_q(D)^N,\\
	\CX^1_q(D) &= L_q(D)^{N+1+N},\\
	\CX^2_q(D) &= L_q(D)^{\CN_2} \times H^2_q(D), \quad \CN_2 = N+1+N+N^2+N+N^2+N+1,\\  
	\CF_\lambda^1 \bF &= (\nabla d, \lambda^{1/2}d, \bff),\\
	\CF_\lambda^2 \bF &=(\nabla d, \lambda^{1/2} d, \bff, \nabla \bg, \lambda^{1/2} \bg, \nabla^2 k, \nabla \lambda^{1/2}k, \lambda k, \zeta),\\
	\fA_q^0 (D) &= L_q(D)^{N^3+N^2+N+1}, \quad \CR^0_\lambda \rho = (\nabla^3 \rho, \lambda^{1/2}\nabla^2\rho, \lambda \nabla \rho, \lambda^{3/2} \rho),
\end{align*}
where $D$ is a domain of $\R^N$.
The $\CR$-boundedness for the solution operator families for the resolvent problem with variable coefficients $\gamma_j$ $(j =1, 2, 3, 4)$ in the whole-space was studied by \cite[Theorem 3.1]{S2020}.
\begin{equation}\label{whole}
\left\{
\begin{aligned}
	&\lambda \rho + \gamma_1 \dv \bu = d & \quad&\text{in $\R^N$}, \\
	&\lambda \bu - \gamma_4^{-1} \DV\{\gamma_2\bD(\bu) + (\gamma_3 - \gamma_2) \dv \bI + \gamma_1 \Delta) \rho \bI\}=\bff& \quad&\text{in $\R^N$}.
\end{aligned}
\right.
\end{equation}

\begin{prop}\label{prop:w}
Let $1 < q < \infty$, and let $\gamma_*$ and $\gamma^*$ be positive constants stated in Assumption \ref{assumption gamma}. 
Then there exist positive constants $\delta_1 \in (0, 1)$
depending on at most $N$, $q$, $\gamma_*$, and $\gamma^*$,
such that for any positive number $\eta_1$ and for uniformly continuous functions $\gamma_i (x)$ $(i =1, 2, 3, 4)$ on $\R^N$ satisfying
\begin{enumerate}
\renewcommand{\theenumi}{\rm {\roman{enumi}}}
\renewcommand{\labelenumi}{\rm {(\theenumi)}}
\item
$\sup_{x \in \R^N} |\gamma_i(x) - \gamma^0_i| \le \delta_1$ with positive constant $\gamma^0_i \in [\gamma_*, \gamma^*]$,
\item
$\|\nabla \gamma_j\|_{L_\infty(\R^N)} \le \eta_1$ for $j=1, 2, 3$, 
\end{enumerate}
there exists a constant $\lambda_1 \ge 1$ such that the following assertions hold true:

$\thetag1$ 
For any $\lambda \in \C_{+, \lambda_1}$, there exist operators
\begin{align*}
&\CA_0 (\lambda) \in 
{\rm Hol} (\C_{+, \lambda_2}, 
\CL(\CX^1_q(\R^N), H^3_q(\R^N)))\\
&\CB_0 (\lambda) \in 
{\rm Hol} (\C_{+, \lambda_2}, 
\CL(\CX^1_q(\R^N), H^2_q(\R^N)^N))
\end{align*}
such that 
for any $\bF=(d, \bff) \in X^1_q(\R^N)$, 
$(\rho, \bu) = (\CA_0 (\lambda) \CF_\lambda^1 \bF, \CB_0 (\lambda) \CF_\lambda^1 \bF)$
is a unique solutions of problem \eqref{whole}.

$\thetag2$ 
There exists a positive constant $r$ such that
\begin{equation*} \label{rbdd bent}
\begin{aligned}
&\CR_{\CL(\CX^1_q(\R^N), \fA_q^0 (\R^N))}
(\{(\tau \pd_\tau)^n \CR_\lambda^0 \CA_0 (\lambda) \mid 
\lambda \in \C_{+, \lambda_1}\}) 
\leq r,\\
&\CR_{\CL(\CX^1_q(\R^N), \fB_q(\R^N))}
(\{(\tau \pd_\tau)^n \CS_\lambda \CB_0 (\lambda) \mid 
\lambda \in \C_{+, \lambda_1}\}) 
\leq r
\end{aligned}
\end{equation*}
for $n = 0, 1$.
Here, above constant $r$ depend on at most $N$, $q$, $\epsilon$,  
$\gamma_*$, and $\gamma^*$. 

\end{prop}
\begin{remark}
As we mention in the end of subsection \ref{subsec:model}, we need to consider the $\CR$-boundedness for the solution operator corresponding to $\lambda^{3/2} \rho$ to prove Lemma \ref{lem:Rbound V g} below.
\end{remark}

On the other hand, the problem with constant coefficients $\alpha_j$ $(j =1, 2, 3, 4, 5)$ in the half-space was studied by \cite{MM}. 
\begin{equation}\label{r0}
\left\{
\begin{aligned}
	&\lambda \rho + \alpha_1 \dv \bu = d & \quad&\text{in $\R^N_+$}, \\
	&\lambda \bu - \DV\{\alpha_2\bD(\bu) + (\alpha_3 - \alpha_2) \dv \bI + \alpha_1 \Delta) \rho \bI\}=\bff& \quad&\text{in $\R^N_+$},\\
	&\{\alpha_2\bD(\bu) + (\alpha_3 - \alpha_2) \dv \bI + \alpha_4 \Delta) \rho \bI\} \bn - \alpha_5 \Delta' h \bn =\bg' & \quad&\text{on $\R^N_0$},\\
	&\bn \cdot \nabla \rho = k & \quad&\text{on $\R^N_0$},\\
	&\lambda h - \bu \cdot \bn = \zeta & \quad&\text{on $\R^N_0$}.
\end{aligned}
\right.
\end{equation}

In the same manner as in \cite[Theorem 1.3]{S} and \cite[Theorem 2]{MM}, we obtain the following result. 
\begin{prop}\label{main'}
Let $1<q<\infty$. 
Assume that $\alpha_j$ $(j =1, 2, 3, 4, 5)$ are positive constants satisfying 
\begin{equation}\label{condi}
	(\alpha_2+\alpha_3)^2 \neq 4\alpha_1 \alpha_4, \quad \alpha_1 \alpha_4 \neq \alpha_2 \alpha_3.
\end{equation}
Then there exists a constant $\lambda_2 \ge 1$ such that the following assertions hold true:

$\thetag1$ 
For any $\lambda \in \C_{+, \lambda_2}$ there exist operators  
\begin{align*}
&\CA_1 (\lambda) \in 
{\rm Hol} (\C_{+, \lambda_2}, 
\CL(\CX_q^2(\R^N_+), H^3_q(\R^N_+)))\\
&\CB_1 (\lambda) \in 
{\rm Hol} (\C_{+, \lambda_2}, 
\CL(\CX_q^2(\R^N_+), H^2_q(\R^N_+)^N)),\\
&\CC_1 (\lambda) \in 
{\rm Hol} (\C_{+, \lambda_2}, 
\CL(\CX_q^2(\R^N_+), H^3_q(\R^N_+)))
\end{align*}
such that 
for any $\bF=(d, \bff, \bg', k, \zeta) \in X_q^0(\R^N_+)$, 
\begin{equation*}
\rho = \CA_1 (\lambda) \CF_\lambda^2 \bF, \quad
\bu = \CB_1 (\lambda) \CF_\lambda^2 \bF, \quad
h = \CC_1 (\lambda) \CF_\lambda^2 \bF
\end{equation*}
are solutions of problem \eqref{r0}.

$\thetag2$ 
There exists a positive constant $r$ such that
\begin{equation} \label{rbdd}
\begin{aligned}
&\CR_{\CL(\CX_q^2(\R^N_+), \fA_q^0 (\R^N_+))}
(\{(\tau \pd_\tau)^n \CR_\lambda^0 \CA_1 (\lambda) \mid 
\lambda \in \C_{+, \lambda_2}\}) 
\leq r,\\
&\CR_{\CL(\CX_q^2(\R^N_+), \fB_q(\R^N_+))}
(\{(\tau \pd_\tau)^n \CS_\lambda \CB_1 (\lambda) \mid 
\lambda \in \C_{+, \lambda_2}\}) 
\leq r,\\
&\CR_{\CL(\CX_q^2(\R^N_+), \fC_q(\R^N_+))}
(\{(\tau \pd_\tau)^n \CT_\lambda \CC_1 (\lambda) \mid 
\lambda \in \C_{+, \lambda_2}\}) 
\leq r
\end{aligned}
\end{equation}
for $n = 0, 1$.
Here, above constants $\lambda_2$ and $r$ depend solely on $N$, $q$, $\epsilon$, and 
$\alpha_j$ $(j=1, 2, 3, 4, 5)$. 

\end{prop}

\begin{remark}
The assumption \eqref{condi} comes from the existence of $\CR$-bounded solution operator families for the resolvent problem corresponding to the Navier-Stokes-Korteweg system without surface tension (cf. \cite[Theorem 1.3]{S}).
Thanks to this result, it was also possible to discuss the resolvent problem taking surface tension into account in \cite{MM}.
\end{remark}

\section{Problem in the bent half-space}\label{sec.bent}
Let $\Phi : \R^N_x \to \R^N_y$ be a bijection of $C^3$ class and let $\Phi^{-1}$ be its inverse map,
where the subscripts $x$ and $y$ denote their variables.
Assume that 
\begin{align*}
	(\nabla_x \Phi)(x) &= \bA + \bB(x),\\
	(\nabla_y \Phi^{-1})(\Phi(x)) &= \bA_{-1} + \bB_{-1}(x),
\end{align*}
where $\bA$ and $\bA_{-1}$ are constant orthonormal matrices whose $(i, j)^{\rm th}$ components $a_{ij}$ and $A_{ij}$, respectively. 
Furthermore, $\bB(x)$ and $\bB_{-1}(x)$ are matrix-valued functions of $H^2_\infty(\R^N)$ whose $(i, j)^{\rm th}$ components $b_{ij}(x)$ and $B_{ij}(x)$, respectively, and satisfy 
\begin{equation}\label{assumption b}
		\|(\bB, \bB_{-1})\|_{L_\infty(\R^N)} \le M_1, \quad
		\|\nabla (\bB, \bB_{-1})\|_{H^1_\infty(\R^N)} \le M_2,
\end{equation}
with $M_1$ and $M_2$ are constants satisfying $0 < M_1 \le 1/2$ and $M_2 \ge 1$.
Let $\Omega_+ = \Phi(\R^N_+)$ and $\Gamma_+ = \Phi(\R^N_0)$, and let $\tbn$ be the unit outer normal to $\Gamma_+$.
Setting $\Phi^{-1} = (\Phi_{-1, 1}, \dots, \Phi_{-1, N})^\mathsf{T}$, 
$\Gamma_+$ is represented by $\Phi_{-1,N}(y)=0$, which yields that
\[
	\tbn(\Phi(x)) = - \frac{\nabla_y \Phi_{-1, N}}{|\nabla_y \Phi_{-1, N}|}
	= - \frac{(A_{N1}+B_{N1}(x), \dots, A_{NN}+B_{NN}(x))^\mathsf{T}}{\{\sum^N_{j=1}(A_{Nj}+B_{Nj}(x))^2\}^{1/2}}
	= \frac{\bA_{\Phi}\bn_0}{|\bA_{\Phi}\bn_0|}
\]
where $\bA_{\Phi} = (\bA_{-1} + \bB_{-1}(x))^\mathsf{T}$, and $\bn_0 = (0, \dots, 0, -1)^\mathsf{T}$.

Next, we introduce the Laplace-Beltrami operator on $\Gamma_+$.
Let the first fundamental form of $\Gamma_+$ and its inverse be $(N-1) \times (N-1)$ matrices $G$ and $G^{-1}$ with $(i, j)^{\rm th}$ components $g_{ij}$ and $g^{ij}$, respectively.
Set $g_+ = \sqrt{\det G}$.
Since $\Gamma_+$ is represented by $y = \Phi(x', 0)$ for $x'=(x_1, \dots, x_{N-1}) \in \R^{N-1}$, 
we see that
\begin{align*}
 g_{ij}(x', 0) &= \frac{\pd}{\pd x_i}\Phi(x', 0) \cdot \frac{\pd}{\pd x_j}\Phi(x', 0)
	= \sum^N_{k=1}(a_{ki} + b_{ki}(x', 0))(a_{kj} + b_{kj}(x', 0))\\
	&= \delta_{ij} + \widetilde g_{ij}(x', 0)
\end{align*}
where $\widetilde g_{ij} = \sum^N_{k=1} (a_{ki}b_{kj} + a_{kj}b_{ki} + b_{ki}b_{kj})$.
The assumption \eqref{assumption b} provides that
\[
	g_+(x', 0) = 1+ \widetilde g_+(x', 0), \quad g^{ij}(x', 0) = \delta_{ij} + \widetilde g^{ij}(x', 0)
\]
with
\begin{equation}\label{1st fundamental}
	\|(\widetilde g_{ij}, \widetilde g^{ij}, \widetilde g_+)\|_{L_\infty(\R^N)} \le C_N M_1,
	\enskip
	\|\nabla (\widetilde g^{ij}, \widetilde g_+)\|_{H^1_\infty(\R^N)} \le C_{M_2}. 
\end{equation}
Set $F(x) = f(\Phi(x))$.
Then the Laplace-Beltrami operator $\Delta_{\Gamma_+}$ on $\Gamma_+$ is given by
\begin{align*}
	(\Delta_{\Gamma_+} f)(\Phi(x', 0)) &= \sum_{i, j=1}^{N-1} \frac{1}{g_+(x', 0)} \frac{\pd}{\pd x_i}
	\left(g_+(x', 0) g^{ij}(x', 0) \frac{\pd}{\pd x_j}F(x', 0)\right)\\
	&= \Delta' F(x', 0) + \CD_+ F(x', 0),
\end{align*}
where
\[
	\CD_+ F(x', 0) = \sum^{N-1}_{i, j=1} \widetilde g^{ij}(x', 0) \frac{\pd^2 F}{\pd x_i \pd x_j}(x', 0) + 
	\sum^{N-1}_{j=1} g^j(x', 0) \frac{\pd F}{\pd x_j}(x', 0)
\] 
with
\[
	g^j(x', 0) = \frac{1}{g_+(x', 0)} \sum^{N-1}_{i=1} \frac{\pd}{\pd x_i} (g_+(x', 0) g^{ij}(x', 0)).
\]
Here \eqref{1st fundamental} yields that
\begin{equation}\label{est:D+}
	\|\CD_+ F\|_{H^1_q(\R^N_+)} \le C_N M_1 \|\nabla^3 F\|_{L_q(\R^N_+)} + C_{M_2} \|(\nabla^2 F, \nabla F)\|_{L_q(\R^N_+)}.
\end{equation}

In this section, we consider the $\CR$-boundedness for the solution operators to the following problem:
\begin{equation}\label{r1}
\left\{
\begin{aligned}
	&\lambda \trho + \tgamma_1 \dv \tbu = \widetilde d & \quad&\text{in $\Omega_+$}, \\
	&\lambda \tbu - \tgamma_4^{-1} 
	\DV\{\tgamma_2 \bD(\tbu) + (\tgamma_3 - \tgamma_2) \dv \tbu \bI + \tgamma_1  \Delta \trho \bI\}=\widetilde \bff& \quad&\text{in $\Omega_+$},\\
	&\{\tgamma_2 \bD(\tbu) + (\tgamma_3 - \tgamma_2) \dv \tbu \bI + \tgamma_1  \Delta \trho \bI \} \tbn - \sigma \Delta_{\Gamma_+} \tth \tbn =\widetilde \bg & \quad&\text{on $\Gamma_+$},\\
	&\tbn \cdot \nabla \trho = \widetilde k & \quad&\text{on $\Gamma_+$},\\
	&\lambda \tth - \tbu \cdot \tbn = \widetilde \zeta & \quad&\text{on $\Gamma_+$}.
\end{aligned}
\right.
\end{equation}

\begin{thm}\label{thm:main bent}
Let $1 < q < \infty$, and let $M_1$ and $M_2$ be the constants given in \eqref{assumption b}.
Assume that $\gamma_*$ and $\gamma^*$ are positive constants stated in Assumption \ref{assumption gamma}. 
Then there exist positive constants $\delta_2 \in (0, 1)$ and $M_1 \in (0, 1/2)$,
depending on at most $N$, $q$, $\gamma_*$, and $\gamma^*$,
such that for any positive number $\eta_2$ and for uniformly continuous functions $\tgamma_i (y)$ $(i =1, 2, 3, 4)$ on $\Omega_+$ satisfying
\begin{enumerate}
\renewcommand{\theenumi}{\rm {\roman{enumi}}}
\renewcommand{\labelenumi}{\rm {(\theenumi)}}
\item
$\sup_{y \in \Omega_+} |\tgamma_i(y) - \gamma^0_i| \le \delta_2$ with positive constant $\gamma^0_i \in [\gamma_*, \gamma^*]$ satisfying
\[
	(\gamma_2^0 + \gamma_3^0)^2 \neq 4 (\gamma_1^0)^2 \gamma_4^0, \quad
	(\gamma_1^0)^2 \gamma_4^0 \neq \gamma_2^0 \gamma_3^0.
\]
\label{condi:gamma1}
\item
$\|\nabla \tgamma_j\|_{L_\infty(\Omega_+)} \le \eta_2$ for $j=1, 2, 3$, \label{condi:gamma2}
\end{enumerate}
there exists a constant $\lambda_2 \ge 1$ such that the following assertions hold true:

$\thetag1$
For any $\lambda \in \C_{+, \lambda_2}$, there exist operators
\begin{align*}
&\CA_2 (\lambda) \in 
{\rm Hol} (\C_{+, \lambda_2}, 
\CL(\CX_q^2(\Omega_+), H^3_q(\Omega_+)))\\
&\CB_2 (\lambda) \in 
{\rm Hol} (\C_{+, \lambda_2}, 
\CL(\CX_q^2(\Omega_+), H^2_q(\Omega_+)^N)),\\
&\CC_2 (\lambda) \in 
{\rm Hol} (\C_{+, \lambda_2}, 
\CL(\CX_q^2(\Omega_+), H^3_q(\Omega_+)))
\end{align*}
such that 
for any $\widetilde \bF=(\widetilde d, \widetilde \bff, \widetilde \bg, \widetilde k, \widetilde \zeta) \in X_q^0(\Omega_+)$, 
\begin{equation*}
\trho = \CA_2 (\lambda) \CF_\lambda^2 \widetilde \bF, \quad
\tbu = \CB_2 (\lambda) \CF_\lambda^2 \widetilde \bF, \quad
\tth = \CC_2 (\lambda) \CF_\lambda^2 \widetilde \bF
\end{equation*}
are solutions of problem \eqref{r1}.

$\thetag2$
There exists a positive constant $r$ such that
\begin{equation} \label{rbdd bent}
\begin{aligned}
&\CR_{\CL(\CX_q^2(\Omega_+), \fA_q^0 (\Omega_+))}
(\{(\tau \pd_\tau)^n \CR_\lambda^0 \CA_2 (\lambda) \mid 
\lambda \in \C_{+, \lambda_2}\}) 
\leq r,\\
&\CR_{\CL(\CX_q^2(\Omega_+), \fB_q(\Omega_+))}
(\{(\tau \pd_\tau)^n \CS_\lambda \CB_2 (\lambda) \mid 
\lambda \in \C_{+, \lambda_2}\}) 
\leq r,\\
&\CR_{\CL(\CX_q^2(\Omega_+), \fC_q(\Omega_+))}
(\{(\tau \pd_\tau)^n \CT_\lambda \CC_2 (\lambda) \mid 
\lambda \in \C_{+, \lambda_2}\}) 
\leq r
\end{aligned}
\end{equation}
for $n = 0, 1$.
Here, above constant $r$ depend on at most $M_2$, $N$, $q$, $\epsilon$,  
$\gamma_*$, $\gamma^*$, and $\sigma$. 

\end{thm}

\subsection{Reduction to a half-space problem}
First, we rewrite \eqref{r1} as follows:
\begin{equation}\label{r2}
\left\{
\begin{aligned}
	&\lambda \trho + \gamma_1^0 \dv \tbu + \CF^1(\tbu) = \widetilde d & \quad&\text{in $\Omega_+$}, \\
	&\lambda \tbu - (\gamma_4^0)^{-1} 
	(\gamma_2^0 \Delta \tbu + \gamma_3^0 \nabla \dv \tbu + \gamma_1^0 \nabla \Delta \trho) - \CF^2(\tbu) - \CF^3(\trho)=\widetilde \bff& \quad&\text{in $\Omega_+$},\\
	&\{\gamma_2^0 \bD(\tbu) + (\gamma_3^0 - \gamma_2^0) \dv \tbu \bI + \gamma_1^0 \Delta \trho \bI + \CF^4(\tbu) + \CF^5(\trho) \} \tbn - \sigma \Delta_{\Gamma_+} \tth \tbn =\widetilde \bg & \quad&\text{on $\Gamma_+$},\\
	&\tbn \cdot \nabla \trho = \widetilde k & \quad&\text{on $\Gamma_+$},\\
	&\lambda \tth - \tbu \cdot \tbn = \widetilde \zeta & \quad&\text{on $\Gamma_+$},
\end{aligned}
\right.
\end{equation}
where
\begin{align*}
	\CF^1(\tbu) &= (\tgamma_1 - \gamma_1^0) \dv \tbu,\\
	\CF^2(\tbu) &= \left(\frac{\tgamma_2}{\tgamma_4} - \frac{\gamma_2^0}{\gamma_4^0}\right) \Delta \tbu 
+ \left(\frac{\tgamma_3}{\tgamma_4} - \frac{\gamma_3^0}{\gamma_4^0}\right)\nabla \dv \tbu
	+ \tgamma_4^{-1} \{\bD(\tbu) \nabla \tgamma_2 + (\dv \tbu) \nabla (\tgamma_3 - \tgamma_2)\},\\	
	\CF^3(\trho) &= \left(\frac{\tgamma_1}{\tgamma_4} - \frac{\gamma_1^0}{\gamma_4^0}\right) \nabla \Delta \trho + \tgamma_4^{-1} (\nabla \tgamma_1) \Delta \trho,\\
	\CF^4(\tbu) &= (\tgamma_2 - \gamma_2^0) \bD(\tbu) + \{\tgamma_3 - \gamma_3^0 - (\tgamma_2 - \gamma_2^0)\} \dv \tbu \bI,\\
	\CF^5(\trho) &= (\tgamma_1 - \gamma_1^0) \Delta \trho \bI.
\end{align*}

Next, we reduce \eqref{r2} to a half-space problem by the change of variable $y = \Phi(x)$.
Let $\bu(x) = \tbu(\Phi(x))$.
The relation
\[
	\frac{\pd}{\pd y_j} = \sum^N_{k=1}(A_{kj} + B_{kj}(x)) \frac{\pd}{\pd x_k},
\enskip
	\nabla_y = (\bA_{-1} + \bB_{-1}(x))^\mathsf{T} \nabla_x
\]
furnishes that
\begin{align*}
	\dv_y \tbu &= \dv_x(\bA_{-1} \bu) + \bB_{-1}(x)^\mathsf{T} : \nabla_x \bu,\\
	\nabla_y \dv_y \tbu &= \bA_\Phi \nabla_x (\dv_x(\bA_{-1} \bu) + \bB_{-1}(x)^\mathsf{T} : \nabla_x \bu),\\
	\bD_y(\tbu) &= \bA_\Phi \nabla_x \bu + (\nabla_x \bu)^\mathsf{T} \bA_\Phi^\mathsf{T},\\
	\Delta_y \tbu &= \Delta_x \bu + \sum_{k, \ell, m =1}^N \left(C_{k \ell m}^1(x) \frac{\pd \bu}{\pd x_m} + C_{k \ell m}^2(x) \frac{\pd^2 \bu}{\pd x_m \pd x_\ell}\right),
\end{align*}
where
\begin{align*}
	C_{k \ell m}^1(x) &= (A_{\ell k} + B_{\ell k}(x))\frac{\pd B_{mk}}{\pd x_\ell}(x),\\
	C_{k \ell m}^2(x) &= (A_{\ell k}B_{mk}(x) + A_{mk}B_{\ell k}(x) + B_{\ell k}(x)B_{mk}(x)).
\end{align*}
Set
\begin{equation}\label{def in half}
\begin{aligned}
	&\rho(x) = \trho(\Phi(x)), &\enskip 
	&\bv(x) = \bA_{-1} \bu(x) = \bA_{-1} \tbu(\Phi(x)), &\enskip
	&h(x) = \tth(\Phi(x)), \\
	&d(x) = \widetilde d(\Phi(x)), &\enskip 
	&\bff(x) = \bA_{-1} \widetilde \bff(\Phi(x)), &\enskip
	&\bg(x) = |\bA_\Phi \bn_0| \bA_{-1} \widetilde \bg(\Phi(x))\\
	&k(x) = \widetilde k(\Phi(x)), &\enskip 
	&\zeta(x) = \widetilde \zeta(\Phi(x)).&
\end{aligned}
\end{equation}
Then \eqref{r2} is reduced to
\begin{equation}\label{r3}
\left\{
\begin{aligned}
	&\lambda \rho + \gamma_1^0 \dv \bv -\CU^1(\bv) = d & \quad&\text{in $\R^N_+$}, \\
	&\lambda \bv - (\gamma_4^0)^{-1} 
	(\gamma_2^0 \Delta \bv + \gamma_3^0 \nabla \dv \bv + \gamma_1^0 \nabla \Delta \rho) - \CU^2(\rho, \bv) = \bff& \quad&\text{in $\R^N_+$},\\
	&\{\gamma_2^0 \bD(\bv) + (\gamma_3^0 - \gamma_2^0) \dv \bv \bI + \gamma_1^0 \Delta \rho \bI \} \bn_0 - \sigma \Delta' h \bn_0 - \CU^3(\rho, \bv, h)= \bg & \quad&\text{on $\R^N_0$},\\
	&\bn_0 \cdot \nabla \rho - \CU^4(\rho) = k & \quad&\text{on $\R^N_0$},\\
	&\lambda h - \bv \cdot \bn_0 - \CU^5(\bv) = \zeta & \quad&\text{on $\R^N_0$},
\end{aligned}
\right.
\end{equation}
where
\begin{align*}
	\CU^1(\bv)&= -\CF^1(\bA_{-1}^\mathsf{T} \bv) - \CG^1(\bv),\\
	\CU^2(\rho, \bv)&= \bA_{-1} \CF^2(\bA_{-1}^\mathsf{T} \bv) + \bA_{-1} \CF^3(\rho) + (\gamma_4^0)^{-1} (\CG^2(\rho) + \CG^3(\bv)),\\	
	\CU^3(\rho, \bv, h)&= (\bA_{-1} \CF^4(\bA_{-1}^\mathsf{T} \bv) + \bA_{-1} \CF^5(\rho)) \bA_\Phi \bn_0 + (\CG^4(\rho) + \CG^5_\bD (\bv) + \CG^5_\dv (\bv) + \CG^6(h))\bn_0,\\
	\CU^4(\rho)&= \bn_0 \cdot \nabla \rho - \frac{\bA_\Phi \bn_0}{|\bA_\Phi \bn_0|} \cdot (\bA_{-1} + \bB_{-1}(x))^\mathsf{T} \nabla \rho,\\
	\CU^5(\bv)&= -\left(1-\frac{1}{|\bA_\Phi \bn_0|}\right)\bv \cdot \bn_0 + \frac{1}{|\bA_\Phi \bn_0|} (\bB_{-1}(x) \bA_{-1}^\mathsf{T} \bv)\cdot \bn_0
\end{align*}
with
\begin{align*}
	\CG^1(\bv)&= \gamma_1^0 (\bB_{-1}(x)^\mathsf{T} : (\bA_{-1}^\mathsf{T} \nabla_x \bv)),\\
	\CG^2(\rho)&= \gamma_1^0 \left[(\bI + \bA_{-1} \bB_{-1}(x)^\mathsf{T}) 
	\nabla \left\{\Delta \rho + \sum_{k, \ell, m =1}^N \left(C_{k \ell m}^1(x) \frac{\pd \rho}{\pd x_m} + C_{k \ell m}^2(x) \frac{\pd^2 \rho}{\pd x_m \pd x_\ell}\right)\right\} 
	- \Delta \nabla \rho \right],\\	
	\CG^3(\bv)&= \gamma_2^0 \sum_{k, \ell, m =1}^N \left(C_{k \ell m}^1(x) \frac{\pd \bv}{\pd x_m} + C_{k \ell m}^2(x) \frac{\pd^2 \bv}{\pd x_m \pd x_\ell}\right) \\
	&\enskip +\gamma_3^0 \{(\bI + \bA_{-1} \bB_{-1}(x)^\mathsf{T}) \nabla(\dv \bv + \bB_{-1}(x)^\mathsf{T} : (\bA_{-1}^\mathsf{T} \nabla \bv)) - \nabla \dv \bv\},\\
	\CG^4(\rho)&= \gamma_1^0 \left[ (\bI + \bA_{-1} \bB_{-1}(x)^\mathsf{T}) \left\{ \Delta \rho + \sum_{k, \ell, m =1}^N \left(C_{k \ell m}^1(x) \frac{\pd \rho}{\pd x_m} + C_{k \ell m}^2(x) \frac{\pd^2 \rho}{\pd x_m \pd x_\ell}\right)\right\} - \Delta \rho\right], \\
	\CG^5_\bD (\bv)&= -\gamma_2^0 \{(\bD(\bv) \bA_{-1} + (\nabla \bv) \bB_{-1}(x) + \bA_{-1} \bB_{-1}(x)^\mathsf{T} (\nabla \bv)^\mathsf{T} \bA_{-1}) \bA_\Phi - \bD(\bv)\},\\
	\CG^5_\dv (\bv)&= -(\gamma_3^0 - \gamma_2^0) \{(\dv \bv + (\bB_{-1}(x)^\mathsf{T} : \bA_{-1}^\mathsf{T} \nabla \bv))( \bI + \bA_{-1} \bB_{-1}(x)^\mathsf{T}) - \dv \bv\},\\
	\CG^6(h)&= \sigma \{( \Delta' h + \CD_+ h) (\bI + \bA_{-1} \bB_{-1}(x)^\mathsf{T}) - \Delta' h\}.
\end{align*}

\subsection{The $\CR$-boundedness for the solution operator families for \eqref{r3}}\label{sec:rbdd half}

In this subsection,
we prove the existence of the $\CR$-boundedness of the solution operator families for \eqref{r3}.
At this point, we recall that the following properties concerning the $\CR$-boundedness that are used throughout this paper. 
These properties were proved by \cite[Proposition 3.4, Remark 3.2]{DHP}.

\begin{lem}\label{lem:5.3}
\begin{enumerate}
\item \label{lem:Rbound1}
Let $X$ and $Y$ be Banach spaces, 
and let $\CT$ and $\CS$ be $\CR$-bounded families in $\CL(X, Y)$. 
Then $\CT+\CS=\{T+S \mid T\in \CT, S\in \CS\}$ is also 
$\CR$-bounded family in $\CL(X, Y)$ and 
\[
\CR_{\CL(X, Y)}(\CT+\CS)\leq \CR_{\CL(X, Y)}(\CT)
+\CR_{\CL(X, Y)}(\CS).
\]
\item \label{lem:Rbound2}
Let $X$, $Y$ and $Z$ be Banach spaces and 
let $\CT$ and $\CS$ be $\CR$-bounded families
 in $\CL(X, Y)$ and $\CL(Y, Z)$, respectively. 
Then $\CS\CT=\{ST \mid T\in \CT, S\in \CS\}$ is also 
an $\CR$-bounded family 
in $\CL(X, Z)$ and 
\[
\CR_{\CL(X, Z)}(\CS\CT)\leq \CR_{\CL(X, Y)}(\CT)\CR_{\CL(Y, Z)}(\CS). 
\]
\item \label{lem:Rbound3}
Let $1 < q < \infty$ and let $D$ be domain in $\R^N$. 
Let $m(\lambda)$ be a bounded function 
defined on a subset $\Lambda$ 
in a complex plane $\C$, and let $M_m(\lambda)$ 
be a multiplication operator with 
$m(\lambda)$ defined by $M_m(\lambda)f=m(\lambda)f$ 
for any $f\in L_q(D)$. 
Then
\[
	\CR_{\CL(L_q(D))}(\{M_m(\lambda) \mid \lambda \in \Lambda\})\leq 
	K_q^2 \sup_{\lambda \in \Lambda} |m(\lambda)|,
\]
where $K_q$ is a positive constant appearing in Khintchine's inequality.
\end{enumerate}
\end{lem}

In the following, we estimate the remainder terms of \eqref{r3}.

\begin{lem}\label{lem:FG}
Let $1< q <\infty$, and let $M_1$ and $M_2$ be the constants given in \eqref{assumption b}.
Assume that the conditions (\ref{condi:gamma1}) and (\ref{condi:gamma2}) in Theorem \ref{thm:main bent} hold for positive constants $\delta_2$ and $\eta_2$ with
\begin{equation}\label{condi:delta}
	0 < \delta_2 < \min \left\{1, \frac{\gamma_*}{2}\right\}.
\end{equation}
Then there exist positive constants $C^1$, $C^2_{M_2, \eta_2}$, and $C^3_{M_2}$ such that
for any $\rho \in H^3_q(\R^N_+)$, $\bv \in H^2_q(\R^N_+)^N$, $h \in H^3_q(\R^N_+)$, and $\lambda \in \C$
the following estimates hold true:
\begin{align*}
	\|\nabla \CU^1(\bv)\|_{L_q(\R^N_+)}& \le C^1(\delta_2 + M_1) \|\nabla^2 \bv\|_{L_q(\R^N_+)} + (C^2_{M_2, \eta_2} + C^3_{M_2}) \|\nabla \bv\|_{L_q(\R^N_+)},\\
	\|\lambda^{1/2} \CU^1(\bv)\|_{L_q(\R^N_+)}& \le C^1(\delta_2 + M_1) \|\lambda^{1/2} \nabla \bv\|_{L_q(\R^N_+)},\\
	\|\CU^2(\rho, \bv)\|_{L_q(\R^N_+)}& \le C^1(\delta_2 + M_1) \|(\nabla^3 \rho, \nabla^2 \bv)\|_{L_q(\R^N_+)} + (C^2_{M_2, \eta_2} + C^3_{M_2}) \|(\nabla \bv, \nabla \rho, \nabla^2 \rho)\|_{L_q(\R^N_+)},\\
	\|\nabla \CU^3(\rho, \bv, h)\|_{L_q(\R^N_+)}& \le C^1(\delta_2 + 2M_1) \|(\nabla^3 \rho, \nabla^2 \bv)\|_{L_q(\R^N_+)} + C^1 M_1 \|\nabla^3 h\|_{L_q(\R^N_+)} \\
	&\enskip + (C^2_{M_2, \eta_2} + 2C^3_{M_2}) \|(\nabla \bv, \nabla \rho, \nabla^2 \rho, \nabla h, \nabla^2 h)\|_{L_q(\R^N_+)},\\
	\|\lambda^{1/2} \CU^3(\rho, \bv, h)\|_{L_q(\R^N_+)}& \le C^1(\delta_2 + 2M_1) \|\lambda^{1/2} (\nabla^2 \rho, \nabla \bv)\|_{L_q(\R^N_+)} + C^1 M_1 \|\lambda^{1/2} \nabla^2 h\|_{L_q(\R^N_+)} \\
	 &\enskip + (C^2_{M_2, \eta_2} + 2C^3_{M_2}) \|\lambda^{1/2} (\nabla \rho, \nabla h)\|_{L_q(\R^N_+)},\\	
	\|\nabla^2 \CU^4(\rho)\|_{L_q(\R^N_+)} &\le C^1 M_1 \|\nabla^3 \rho\|_{L_q(\R^N_+)} + C^3_{M_2} \|(\nabla \rho, \nabla^2 \rho)\|_{L_q(\R^N_+)},\\
	\|\lambda^{1/2} \nabla \CU^4(\rho)\|_{L_q(\R^N_+)} &\le C^1 M_1 \|\lambda^{1/2} \nabla^2 \rho\|_{L_q(\R^N_+)} + C^3_{M_2} \|\lambda^{1/2} \nabla \rho\|_{L_q(\R^N_+)},\\
	\|\lambda \CU^4(\rho)\|_{L_q(\R^N_+)} &\le C^1 M_1 \|\lambda \nabla \rho\|_{L_q(\R^N_+)},\\
	\|\nabla^2 \CU^5(\bv)\|_{L_q(\R^N_+)} &\le C^1 M_1 \|\nabla^2 \bv\|_{L_q(\R^N_+)} + C^3_{M_2} \|(\nabla \bv, \bv)\|_{L_q(\R^N_+)},\\	
	\|\nabla \CU^5(\bv)\|_{L_q(\R^N_+)} &\le C^1 M_1 \|\nabla \bv\|_{L_q(\R^N_+)} + C^3_{M_2} \|\bv\|_{L_q(\R^N_+)},\\	
	\|\CU^5(\bv)\|_{L_q(\R^N_+)} &\le C^1 M_1 \|\bv\|_{L_q(\R^N_+)},
\end{align*}
where $C^1$ depends on $N$, $q$, $\gamma_*$,  and $\gamma^*$, but is independent of $M_1$, $M_2$, $\delta_2$, and $\eta_2$;
$C^2_{M_2, \eta_2}$ depends on $M_2$, $\eta_2$, $N$, $q$, $\gamma_*$,  and $\gamma^*$, but is independent of $M_1$ and $\delta_2$;
$C^3_{M_2}$ depends on $M_2$, $N$, $q$, $\gamma_*$,  and $\gamma^*$, but is independent of $M_1$, $\delta_2$, and $\eta_2$.
\end{lem}
\begin{proof}
Note that\eqref{condi:delta} implies that $\gamma_*/2 \le \tgamma_4(y) \le \gamma^* + 1$ for $y \in \Omega_+$. 
According to the definition of $\CU^i$ for $i = 1, 2, 3$, it suffices to consider the estimates of $\CF^j$ and $\CG^k$ for $j = 1, \dots, 5$ and for $k = 1, \dots, 6$
which can be obtained by the same calculation as in the proof of \cite[Lemma 5.3]{S2020}:
\begin{equation}\label{est:F}
\begin{aligned}
	\|\nabla \CF^1(\bA_{-1}^\mathsf{T} \bv)\|_{L_q(\R^N_+)} &\le C^1 \delta_2 \|\nabla^2 \bv\|_{L_q(\R^N_+)} + C^2_{M_2, \eta_2} \|\nabla \bv\|_{L_q(\R^N_+)},\\
	\|\lambda^{1/2} \CF^1(\bA_{-1}^\mathsf{T} \bv)\|_{L_q(\R^N_+)} &\le C^1 \delta_2 \|\lambda^{1/2} \nabla \bv\|_{L_q(\R^N_+)},\\
	\|\bA_{-1} \CF^2(\bA_{-1}^\mathsf{T} \bv)\|_{L_q(\R^N_+)} &\le C^1 \delta_2 \|\nabla^2 \bv\|_{L_q(\R^N_+)} + C^2_{M_2, \eta_2} \|\nabla \bv\|_{L_q(\R^N_+)},\\
	\|\bA_{-1} \CF^3(\rho) \|_{L_q(\R^N_+)} &\le C^1 \delta_2 \|\nabla^3 \rho\|_{L_q(\R^N_+)} + C^2_{M_2, \eta_2} \|\nabla^2 \rho\|_{L_q(\R^N_+)},\\
	\|\nabla (\CF^4(\bv) \bA_\Phi)\|_{L_q(\R^N_+)} &\le C^1 \delta_2 \|\nabla^2 \bv\|_{L_q(\R^N_+)} + C^2_{M_2, \eta_2} \|\nabla \bv\|_{L_q(\R^N_+)},\\
	\|\lambda^{1/2} \bA_{-1} \CF^4(\bA_{-1}^\mathsf{T} \bv) \bA_\Phi \|_{L_q(\R^N_+)} &\le C^1 \delta_2 \|\lambda^{1/2} \nabla \bv\|_{L_q(\R^N_+)},\\
	\|\nabla (\CF^5(\rho) \bA_\Phi)\|_{L_q(\R^N_+)} &\le C^1 \delta_2 \|\nabla^3 \rho\|_{L_q(\R^N_+)} + C^2_{M_2, \eta_2} \|\nabla^2 \rho\|_{L_q(\R^N_+)},\\
	\|\lambda^{1/2} \CF^5(\rho) \bA_\Phi \|_{L_q(\R^N_+)} &\le C^1 \delta_2 \|\lambda^{1/2} \nabla^2 \rho\|_{L_q(\R^N_+)},
\end{aligned}
\end{equation}
and also
\begin{equation}\label{est:G}
\begin{aligned}
	\|\nabla \CG^1(\bv)\|_{L_q(\R^N_+)} &\le C^1 M_1 \|\nabla^2 \bv\|_{L_q(\R^N_+)} + C^3_{M_2} \|\nabla \bv\|_{L_q(\R^N_+)},\\
	\|\lambda^{1/2} \CG^1(\bv)\|_{L_q(\R^N_+)} &\le C^1 M_1 \|\lambda^{1/2} \nabla \bv\|_{L_q(\R^N_+)},\\
	\|(\gamma_4^0)^{-1} \CG^2(\rho)\|_{L_q(\R^N_+)} &\le C^1 M_1 \|\nabla^3 \rho\|_{L_q(\R^N_+)} + C^3_{M_2} \|(\nabla \rho, \nabla^2 \rho)\|_{L_q(\R^N_+)},\\
	\|(\gamma_4^0)^{-1} \CG^3(\bv)\|_{L_q(\R^N_+)} &\le C^1 M_1 \|\nabla^2 \bv\|_{L_q(\R^N_+)} + C^3_{M_2} \|\nabla \bv\|_{L_q(\R^N_+)},\\
	\|\nabla \CG^4(\rho)\|_{L_q(\R^N_+)} &\le C^1 M_1 \|\nabla^3 \rho\|_{L_q(\R^N_+)} + C^3_{M_2} \|(\nabla \rho, \nabla^2 \rho)\|_{L_q(\R^N_+)},\\
	\|\lambda^{1/2} \CG^4(\rho)\|_{L_q(\R^N_+)} &\le C^1 M_1 \|\lambda^{1/2} \nabla^2 \rho\|_{L_q(\R^N_+)} + C^3_{M_2} \|\lambda^{1/2} \nabla \rho\|_{L_q(\R^N_+)},\\
	\|\nabla (\CG^5_\bD(\bv), \CG^5_\dv(\bv))\|_{L_q(\R^N_+)} &\le C^1 M_1 \|\nabla^2 \bv\|_{L_q(\R^N_+)} + C^3_{M_2} \|\nabla \bv\|_{L_q(\R^N_+)},\\
	\|\lambda^{1/2} (\CG^5_\bD(\bv), \CG^5_\dv(\bv))\|_{L_q(\R^N_+)} &\le C^1 M_1 \|\lambda^{1/2} \nabla \bv\|_{L_q(\R^N_+)},\\
	\|\nabla \CG^6(h)\|_{L_q(\R^N_+)} &\le C^1 M_1 \|\nabla^3 h\|_{L_q(\R^N_+)} + C^3_{M_2} \|(\nabla h, \nabla^2 h)\|_{L_q(\R^N_+)},\\
	\|\lambda^{1/2} \CG^6(h)\|_{L_q(\R^N_+)} &\le C^1 M_1 \|\lambda^{1/2} \nabla^2 h\|_{L_q(\R^N_+)} + C^3_{M_2} \|\lambda^{1/2} \nabla h\|_{L_q(\R^N_+)}.
\end{aligned}
\end{equation}
Since $\CU^4(\rho)$ and $\CU^5(\bv)$ can be estimated by the same manner as the above estimates, this completes the proof of Lemma \ref{lem:FG}.

\end{proof}

Set
\[
	\CU(\rho, \bv, h) = (\CU^1(\bv), \CU^2(\rho, \bv), \CU^3(\rho, \bv, h), \CU^4(\rho), \CU^5(\bv)).
\]

\begin{cor}\label{apriori}
Let $1< q <\infty$, and let $M_1$ and $M_2$ be the constants given in \eqref{assumption b}.
Assume that the conditions (\ref{condi:gamma1}) and (\ref{condi:gamma2}) in Theorem \ref{thm:main bent} hold for positive constants $\delta_2$ and $\eta_2$ with
\eqref{condi:delta}.
Then 
for any $\rho \in H^3_q(\R^N_+)$, $\bv \in H^2_q(\R^N_+)^N$, $h \in H^3_q(\R^N_+)$,
$\lambda_2 \ge 1$, and $\lambda \in \C_{+, \lambda_2}$, the following estimate holds true:
\begin{align*}
	&\|\CF_\lambda^2 \CU(\rho, \bv, h)\|_{\CX_q^2(\R^N_+)} \\
	& \le \{C^1 (5\delta_2 + 11 M_1) + (4C^2_{M_2, \eta_2} + 10 C^3_{M_2}) \lambda_2^{-1/2}\}
	(\|\CR_\lambda^0 \rho\|_{\fA_q^0 (\R^N_+)} + \|\CS_\lambda \bv\|_{\fB_q(\R^N_+)}
	+\|\CT_\lambda h\|_{\fC_q(\R^N_+)}).
\end{align*}
\end{cor}

\begin{proof}
Let $(\rho, \bv, h) \in H^3_q(\R^N_+) \times H^2_q(\R^N_+)^N \times H^3_q(\R^N_+)$.
For any $\lambda_2 \ge 1$ and $\lambda \in \C_{+, \lambda_2}$, Lemma \ref{lem:FG} yields that 
\begin{align*}
	\|\nabla \CU^1(\bv)\|_{L_q(\R^N_+)}& \le C^1(\delta_2 + M_1) \|\nabla^2 \bv\|_{L_q(\R^N_+)} + (C^2_{M_2, \eta_2} + C^3_{M_2}) \lambda_2^{-1/2}\|\lambda^{1/2} \nabla \bv\|_{L_q(\R^N_+)},\\
	\|\lambda^{1/2} \CU^1(\bv)\|_{L_q(\R^N_+)}& \le C^1(\delta_2 + M_1) \|\lambda^{1/2} \nabla \bv\|_{L_q(\R^N_+)},\\
	\|\CU^2(\rho, \bv)\|_{L_q(\R^N_+)}& \le C^1(\delta_2 + M_1) \|(\nabla^3 \rho, \nabla^2 \bv)\|_{L_q(\R^N_+)} \\
	&\enskip + (C^2_{M_2, \eta_2} + C^3_{M_2}) (\lambda_2^{-1/2} \|\lambda^{1/2} (\nabla \bv, \nabla^2 \rho)\|_{L_q(\R^N_+)} + \lambda_2^{-1}\|\lambda \nabla \rho\|_{L_q(\R^N_+)}),\\
	\|\nabla \CU^3(\rho, \bv, h)\|_{L_q(\R^N_+)}& \le C^1(\delta_2 + 2M_1) \|(\nabla^3 \rho, \nabla^2 \bv)\|_{L_q(\R^N_+)} + C^1 M_1 \|\nabla^3 h\|_{L_q(\R^N_+)} \\
	&\enskip + (C^2_{M_2, \eta_2} + 2C^3_{M_2}) (\lambda_2^{-1/2} \|\lambda^{1/2} (\nabla \bv, \nabla^2 \rho)\|_{L_q(\R^N_+)} + \lambda_2^{-1} \|\lambda (\nabla \rho, \nabla h, \nabla^2 h)\|_{L_q(\R^N_+)}),\\
	\|\lambda^{1/2} \CU^3(\rho, \bv, h)\|_{L_q(\R^N_+)}& \le C^1(\delta_2 + 2M_1) \|\lambda^{1/2} (\nabla^2 \rho, \nabla \bv)\|_{L_q(\R^N_+)} + C^1 M_1 \|\lambda \nabla^2 h\|_{L_q(\R^N_+)}\\
	&\enskip + (C^2_{M_2, \eta_2} + 2C^3_{M_2}) \lambda_2^{-1} \|\lambda (\nabla \rho, \nabla h)\|_{L_q(\R^N_+)},\\
	\|\nabla^2 \CU^4(\rho)\|_{L_q(\R^N_+)}& \le C^1 M_1 \|\nabla^3 \rho\|_{L_q(\R^N_+)}
	+ C^3_{M_2} (\lambda_2^{-1} \|\lambda \nabla \rho\|_{L_q(\R^N_+)} + \lambda_2^{-1/2} \|\lambda^{1/2} \nabla^2 \rho\|_{L_q(\R^N_+)}),\\
	\|\lambda^{1/2} \nabla \CU^4(\rho)\|_{L_q(\R^N_+)}& \le C^1 M_1 \|\lambda^{1/2} \nabla^2 \rho\|_{L_q(\R^N_+)}
	+ C^3_{M_2} \lambda_2^{-1/2} \|\lambda \nabla \rho\|_{L_q(\R^N_+)},\\
	\|\lambda \CU^4(\rho)\|_{L_q(\R^N_+)}& \le C^1 M_1 \|\lambda \nabla \rho\|_{L_q(\R^N_+)}\\
	\|\CU^5(\bv)\|_{H^2_q(\R^N_+)}& \le C^1 M_1 (\lambda_2^{-1} \|\lambda \bv\|_{L_q(\R^N_+)} + \lambda_2^{-1/2} \|\lambda^{1/2} \nabla \bv\|_{L_q(\R^N_+)} + \|\nabla^2 \bv\|_{L_q(\R^N_+)}) \\
	&\enskip + 2C^3_{M_2} (\lambda_2^{-1/2} \|\lambda^{1/2} \nabla \bv\|_{L_q(\R^N_+)} + \lambda_2^{-1} \|\lambda \bv\|_{L_q(\R^N_+)}).
\end{align*}
Therefore
\begin{align*}
	&\|\CF_\lambda^2 \CU(\rho, \bv, h)\|_{\CX_q^2(\R^N_+)} \\
	&= \|\CU^1(\bv)\|_{L_q(\R^N_+)} + \|\lambda^{1/2} \CU^1(\bv)\|_{L_q(\R^N_+)} + \|\CU^2(\rho, \bv)\|_{L_q(\R^N_+)}
	+ \|\nabla \CU^3(\rho, \bv, h)\|_{L_q(\R^N_+)} + \|\lambda^{1/2} \CU^3(\rho, \bv, h)\|_{L_q(\R^N_+)}\\ 
	&\enskip + \|\nabla^2 \CU^4(\rho)\|_{L_q(\R^N_+)} + \|\lambda^{1/2} \nabla \CU^4(\rho)\|_{L_q(\R^N_+)} + \|\lambda \CU^4(\rho)\|_{L_q(\R^N_+)}
	+ \|\CU^5(\bv)\|_{H^2_q(\R^N_+)}\\
	& \le \{C^1 (5\delta_2 + 11 M_1) + (4C^2_{M_2, \eta_2} + 10 C^3_{M_2}) \lambda_2^{-1/2}\}
	(\|\CR_\lambda^0 \rho\|_{\fA_q^0 (\R^N_+)} + \|\CS_\lambda \bv\|_{\fB_q(\R^N_+)}
	+\|\CT_\lambda h\|_{\fC_q(\R^N_+)}),
\end{align*}
where we have used $\lambda_2^{-1/2} \le 1$,
which completes the proof of Corollary \ref{apriori}.
\end{proof}

Let 
$\widetilde \bF = (\widetilde d, \widetilde \bff, \widetilde \bg, \widetilde k, \widetilde \zeta) \in X_q^0(\Omega_+)$.
Note that $\bF  = (d, \bff, \bg, k, \zeta) \in X_q^0(\R^N_+)$ by \eqref{assumption b}, then
according to Proposition \ref{main'} with $(\alpha_1, \alpha_2, \alpha_3, \alpha_4, \alpha_5) = (\gamma_1^0, \gamma_2^0/\gamma_4^0, \gamma_3^0/\gamma_4^0, \gamma_1^0/\gamma_4^0, \sigma/\gamma_4^0)$ and $\bg' = \bg/\gamma_4^0$, we see that
$\rho = \CA_1(\lambda) \CF_\lambda^2 \bF$,
$\bv = \CB_1(\lambda) \CF_\lambda^2 \bF$, and $h = \CC_1(\lambda) \CF_\lambda^2 \bF$
satisfy
\begin{equation}\label{r4}
\left\{
\begin{aligned}
	&\lambda \rho + \gamma_1^0 \dv \bv -\CU^1(\bv) = d - \CV^1(\lambda)\CF_\lambda^2 \bF& \quad&\text{in $\R^N_+$}, \\
	&\lambda \bv - (\gamma_4^0)^{-1} 
	(\gamma_2^0 \Delta \bv + \gamma_3^0 \nabla \dv \bv + \gamma_1^0 \nabla \Delta \rho) - \CU^2(\rho, \bv) = \bff - \CV^2(\lambda)\CF_\lambda^2 \bF& \quad&\text{in $\R^N_+$},\\
	&\{\gamma_2^0 \bD(\bv) + (\gamma_3^0 - \gamma_2^0) \dv \bv \bI + \gamma_1^0 \Delta \rho \bI \} \bn_0 - \sigma \Delta' h \bn_0 - \CU^3(\rho, \bv, h)= \bg - \CV^3(\lambda)\CF_\lambda^2 \bF& \quad&\text{on $\R^N_0$},\\
	&\bn_0 \cdot \nabla \rho - \CU^4(\rho) = k - \CV^4(\lambda)\CF_\lambda^2 \bF& \quad&\text{on $\R^N_0$},\\
	&\lambda h - \bv \cdot \bn_0 - \CU^5(\bv) = \zeta - \CV^5(\lambda)\CF_\lambda^2 \bF& \quad&\text{on $\R^N_0$},
\end{aligned}
\right.
\end{equation}
where
\begin{align*}
	\CV^1(\lambda) \CF_\lambda^2 \bF &= \CU^1(\CB_1(\lambda)\CF_\lambda^2 \bF), &\enskip 
	\CV^2(\lambda) \CF_\lambda^2 \bF &= \CU^2(\CA_1(\lambda)\CF_\lambda^2 \bF, \CB_1(\lambda)\CF_\lambda^2 \bF), \\
	\CV^3(\lambda) \CF_\lambda^2 \bF &= \CU^3(\CA_1(\lambda)\CF_\lambda^2 \bF, \CB_1(\lambda)\CF_\lambda^2 \bF, \CC_1(\lambda)\CF_\lambda^2 \bF), &\enskip 
	\CV^4(\lambda) \CF_\lambda^2 \bF &= \CU^4(\CA_1(\lambda)\CF_\lambda^2 \bF),\\
	\CV^5(\lambda) \CF_\lambda^2 \bF &= \CU^5(\CB_1(\lambda)\CF_\lambda^2 \bF). &\enskip &
\end{align*}
Let us define  
\[
	\CV(\lambda) \bG = (\CV^1(\lambda) \bG, \CV^2(\lambda) \bG, \dots \CV^5(\lambda) \bG)
\]
for $\lambda \in \C_{+, \lambda_2}$ and $\bG \in \CX^2_q(\R^N_+)$.
The following lemma will be proved in the next subsection.
\begin{lem}\label{lem:Rbound V}
Let $1< q <\infty$, and let $M_1$ and $M_2$ be the constants given in \eqref{assumption b}.
Assume that the conditions (\ref{condi:gamma1}) and (\ref{condi:gamma2}) in Theorem \ref{thm:main bent} hold for positive constants $\delta_2$ and $\eta_2$ with
\eqref{condi:delta}.
Then there exists a constant $C_q \ge 11$ depending only on $q$ such that for $n=0, 1$ and $\lambda_2 \ge 1$, the following estimate holds true:
\begin{equation}\label{Rbound V}
	\CR_{\CL(\CX_q^2(\R^N_+))}\{(\tau \pd_\tau)^n \CF_\lambda^2 \CV(\lambda) \mid \lambda \in \C_{+, \lambda_2}\}
\le C_q r (C^1 (\delta_2 + M_1) + (C^2_{M_2, \eta_2} + C^3_{M_2}) \lambda_2^{-1/2})
\end{equation}
for $n=0, 1$, where $C_q$ is come positive constant.
\end{lem}
Admitting Lemma \ref{lem:Rbound V} and choosing $\delta_2$ and $M_1$ so small and $\lambda_2$ so large that
\begin{equation}\label{M lambda}
	C_q r C^1(\delta_2 + M_1) \le \frac14, \enskip 
	C_q r C^1(C^2_{M_2, \eta_2} + C^3_{M_2}) \lambda_2^{-1/2} \le \frac14,
\end{equation} 
\eqref{Rbound V} implies that
\begin{equation}\label{Rbound V'}
	\CR_{\CL(\CX_q^2(\R^N_+))}\{(\tau \pd_\tau)^n \CF_\lambda^2 \CV(\lambda) \mid \lambda \in \C_{+, \lambda_2}\}
\le \frac12.
\end{equation}
Thus there exists the inverse operator $(\bI - \CF_\lambda^2 \CV(\lambda))^{-1}$ satisfying
\begin{equation}\label{Rbound inverse V}
	\CR_{\CL(\CX_q^2(\R^N_+))}\{(\tau \pd_\tau)^n (\bI - \CF_\lambda^2 \CV(\lambda))^{-1} \mid \lambda \in \C_{+, \lambda_2}\}
\le 2.
\end{equation}
For any fixed $\lambda \in \C_{+, \lambda_2}$, we introduce the norm
$
	\|\bH\|_{X_{q, \lambda}(\R^N_+)} = \|\CF_\lambda^2 \bH\|_{\CX_q^2(\R^N_+)}
$
for any $\bH \in X_q^0(\R^N_+)$ and also its space 
$
	X_{q, \lambda}(\R^N_+) = (X^0_q(\R^N_+), \|\cdot\|_{X_{q, \lambda}(\R^N_+)})
$.
Then it holds by \eqref{Rbound V'} that
\[
	\|\CV(\lambda) \CF_\lambda^2 \bH\|_{X_{q, \lambda}(\R^N_+)}
	=\|\CF_\lambda^2 \CV(\lambda) \CF_\lambda^2 \bH\|_{\CX_q^2(\R^N_+)}
	\le \frac12 \|\CF_\lambda^2 \bH\|_{\CX_q^2(\R^N_+)}
	= \frac12 \|\bH\|_{X_{q, \lambda}(\R^N_+)},
\]
thus the inverse operator $(\bI - \CV(\lambda) \CF_\lambda^2)^{-1}$ exists for any fixed $\lambda \in \C_{+, \lambda_2}$. 
Note that
\[
	\CF_\lambda^2 (\bI - \CV(\lambda) \CF_\lambda^2)^{-1} = (\bI - \CF_\lambda^2\CV(\lambda))^{-1} \CF_\lambda^2
\]
and set
\begin{align*}
	\CA_3(\lambda) \bG &= \CA_1(\lambda) (\bI - \CF_\lambda^2\CV(\lambda))^{-1} \bG,\\
	\CB_3(\lambda) \bG &= \CB_1(\lambda) (\bI - \CF_\lambda^2\CV(\lambda))^{-1} \bG,\\
	\CC_3(\lambda) \bG &= \CC_1(\lambda) (\bI - \CF_\lambda^2\CV(\lambda))^{-1} \bG
\end{align*}
for $\bG \in \CX_q^2(\R^N_+)$,
then
we observe that 
for any $\bF = (d, \bff, \bg, k, \zeta) \in X_q^0(\R^N_+)$, 
$(\rho, \bv, h) = (\CA_3(\lambda) \CF_\lambda^2 \bF, \CB_3(\lambda) \CF_\lambda^2 \bF, \CC_3(\lambda) \CF_\lambda^2 \bF)$
solves \eqref{r3}. In addition,  Lemma \ref{lem:5.3} (\ref{lem:Rbound2}), \eqref{rbdd}, and \eqref{Rbound inverse V} furnish that
\begin{equation} \label{rbdd half}
\begin{aligned}
&\CR_{\CL(\CX_q^2(\R^N_+), \fA_q^0 (\R^N_+))}
(\{(\tau \pd_\tau)^n \CR_\lambda^0 \CA_3 (\lambda) \mid 
\lambda \in \C_{+, \lambda_2}\}) 
\leq 2r,\\
&\CR_{\CL(\CX_q^2(\R^N_+), \fB_q(\R^N_+))}
(\{(\tau \pd_\tau)^n \CS_\lambda \CB_3 (\lambda) \mid 
\lambda \in \C_{+, \lambda_2}\}) 
\leq 2r,\\
&\CR_{\CL(\CX_q^2(\R^N_+), \fC_q(\R^N_+))}
(\{(\tau \pd_\tau)^n \CT_\lambda \CC_3 (\lambda) \mid 
\lambda \in \C_{+, \lambda_2}\}) 
\leq 2r
\end{aligned}
\end{equation}
for $n = 0, 1$.

The uniqueness of solutions of \eqref{r3} follows from the a priori estimate.
In fact,
assume that $(\rho, \bv, h)$ satisfies \eqref{r3} with $(d, \bff, \bg, k, \zeta)=(0, 0, 0, 0, 0)$,
then by \eqref{rbdd}, Corollary \ref{apriori}, Lemma \ref{lem:Rbound V}, and \eqref{M lambda} 
\begin{align*}
	&\|\CR_\lambda^0 \rho\|_{\fA_q^0 (\R^N_+)} + \|\CS_\lambda \bv\|_{\fB_q(\R^N_+)}
	+\|\CT_\lambda h\|_{\fC_q(\R^N_+)}\\
	&\le r\|\CF_\lambda^2 \CU(\rho, \bv, h)\|_{\CX_q^2(\R^N_+)}\\
	&\le 11 r\{C^1 (\delta_2 + M_1) + (C^2_{M_2, \eta_2} + C^3_{M_2}) \lambda_2^{-1/2}\}
	(\|\CR_\lambda^0 \rho\|_{\fA_q^0 (\R^N_+)} + \|\CS_\lambda \bv\|_{\fB_q(\R^N_+)}
	+\|\CT_\lambda h\|_{\fC_q(\R^N_+)})\\
	& \le C_q r\{C^1(\delta_2 + M_1) + (C^2_{M_2, \eta_2} + C^3_{M_2}) \lambda_2^{-1/2}\}
	(\|\CR_\lambda^0 \rho\|_{\fA_q^0 (\R^N_+)} + \|\CS_\lambda \bv\|_{\fB_q(\R^N_+)}
	+\|\CT_\lambda h\|_{\fC_q(\R^N_+)})\\
	& \le \frac12 (\|\CR_\lambda^0 \rho\|_{\fA_q^0 (\R^N_+)} + \|\CS_\lambda \bv\|_{\fB_q(\R^N_+)}
	+\|\CT_\lambda h\|_{\fC_q(\R^N_+)}),
\end{align*}
which implies $(\rho, \bv, h) = (0, 0, 0)$.

\subsection{Proof of Lemma \ref{lem:Rbound V}}
In this subsection, we prove that there exists a positive constant $C_q \ge 11$ such that \eqref{Rbound V} by Lemma \ref{lem:FG}.
First, we prove
\begin{equation}\label{est:V1}
	\CR_{\CL(L_q(\R^N_+))}(\{(\tau \pd_\tau)^n (\nabla \CV^1(\lambda)) \mid \lambda \in \C_{+, \lambda_2}\})
	\le C_q r (C^1 (\delta_2 + M_1) + (C^2_{M_2, \eta_2} + C^3_{M_2}) \lambda_2^{-1/2})
\end{equation}
according to the definition of the $\CR$-boundedness (cf. Definition \ref{dfn2}).
Since $\nabla \CU^1$ is the linear operator from $H^2_q(\R^N_+)^N$ to $L_q(\R^N_+)^N$ by Lemma \ref{lem:FG},
it holds that for $n=0, 1$ and for any $m \in \N$, $\{\lambda_j\}_{j=1}^m \subset \C_{+, \lambda_2}$, and $\{\bG_j\}_{j=1}^m \subset \CX_q^2(\R^N_+)$,
\begin{align*}
	&\int^1_0 \Big \|\sum^m_{j=1} r_j(u) \{(\tau \pd_\tau)^n (\nabla \CV^1(\lambda)) \}|_{\lambda = \lambda_j} \bG_j \Big \|_{L_q(\R^N_+)}^q \, du\\
	&= \int^1_0 \Big \|\nabla \CU^1\left(\sum^m_{j=1}  r_j(u) \{(\tau \pd_\tau)^n \CB_1(\lambda)\}|_{\lambda = \lambda_j} \bG_j \right) \Big \|_{L_q(\R^N_+)}^q \, du\\
	&\le C_q  \Big{[} \{C^1(\delta_2 + M_1)\}^q \int^1_0 \Big \|\sum^m_{j=1}  r_j(u) \{(\tau \pd_\tau)^n \nabla^2 \CB_1(\lambda)\}|_{\lambda = \lambda_j} \bG_j \Big \|_{L_q(\R^N_+)}^q\, du \\
	&\quad + (C^2_{M_2, \eta_2} + C^3_{M_2})^q \int^1_0 \Big \|\sum^m_{j=1}  r_j(u) \{(\tau \pd_\tau)^n \nabla \CB_1(\lambda)\}|_{\lambda = \lambda_j} \bG_j \Big \|_{L_q(\R^N_+)}^q\, du \Big{]} \\
	&=: I_n^1,
\end{align*}
where $C_q$ is a positive constant depending only on $q$.
In the case $n=0$, Lemma \ref{lem:5.3} (\ref{lem:Rbound3}) and \eqref{rbdd} yield that
\begin{align*}
	I_0^1 &
	\le C_q \Big{[} \{C^1(\delta_2 + M_1)\}^q \int^1_0 \Big \|\sum^m_{j=1}  r_j(u) \nabla^2 \CB_1(\lambda_j) \bG_j \Big \|_{L_q(\R^N_+)}^q\, du \\
	&\quad + (C^2_{M_2, \eta_2} + C^3_{M_2})^q (K_q^2 \sup_{\lambda \in \C_{+, \lambda_2}} |\lambda|^{-1/2})^q \int^1_0 \Big \|\sum^m_{j=1}  r_j(u) \lambda_j^{1/2}\nabla \CB_1(\lambda_j) \bG_j \Big \|_{L_q(\R^N_+)}^q \, du \Big{]}\\
	&\le C_q r^q [\{C^1(\delta_2 + M_1)\}^q + (C^2_{M_2, \eta_2} + C^3_{M_2})^q K_q^{2q} \lambda_2^{-q/2}] \int^1_0 \Big \|\sum^m_{j=1}  r_j(u) \bG_j \Big \|_{\CX_q^2(\R^N_+)}^q\, du.
\end{align*}
In the case $n=1$, note that 
\[
	(\tau \pd_\tau) \nabla \CB_1(\lambda) 
	= \frac{1}{\lambda^{1/2}} \left\{(\tau \pd_\tau)(\lambda^{1/2} \nabla \CB_1(\lambda)) - \frac{i}{2} \frac{\tau}{\lambda}(\lambda^{1/2} \nabla \CB_1(\lambda))\right\}.
\]
Then Lemma \ref{lem:5.3} (\ref{lem:Rbound1}), (\ref{lem:Rbound3}), and \eqref{rbdd} yield that
\begin{align*}
	&\CR_{\CL(\CX_q^2(\R^N_+), L_q(\R^N_+))} \{(\tau \pd_\tau) \nabla \CB_1(\lambda) \mid \lambda \in \C_{+, \lambda_2}\}\\
	&\le r K_q^2 \left(\sup_{\lambda \in \C_{+, \lambda_2}} |\lambda|^{-1/2} + \frac12 \sup_{\lambda \in \C_{+, \lambda_2}} (|\lambda|^{-1/2} |\tau/\lambda|)\right)\\
	&\le \frac{3}{2} r K_q^2 \lambda_2^{-1/2}, 
\end{align*}
therefore
\begin{align*}
	I_1^1 &
	\le C_q r^q [\{C^1(\delta_2 + M_1)\}^q + (C^2_{M_2, \eta_2} + C^3_{M_2})^q (3/2)^q K_q^{2q} \lambda_2^{-q/2}] \int^1_0 \Big \|\sum^m_{j=1}  r_j(u) \bG_j \Big \|_{\CX_q^2(\R^N_+)}^q\, du\\
	& \le C_q r^q [\{C^1(\delta_2 + M_1)\}^q + (C^2_{M_2, \eta_2} + C^3_{M_2})^q K_q^{2q} \lambda_2^{-q/2}] \int^1_0 \Big \|\sum^m_{j=1}  r_j(u) \bG_j \Big \|_{\CX_q^2(\R^N_+)}^q\, du,
\end{align*}
where $C_q$ may change from line to line. 
Combining $I_0^1$ and $I_1^1$, we have
\begin{align*}
	&\left(\int^1_0 \Big \|\sum^m_{j=1} r_j(u) \{(\tau \pd_\tau)^n (\nabla \CV^1(\lambda)) \}|_{\lambda = \lambda_j} \bG_j \Big \|_{L_q(\R^N_+)}^q \, du \right)^{1/q}\\
	& \le C_q r \{C^1(\delta_2 + M_1) + (C^2_{M_2, \eta_2} + C^3_{M_2}) \lambda_2^{-1/2}\} \left(\int^1_0 \Big \|\sum^m_{j=1}  r_j(u) \bG_j \Big \|_{\CX_q^2(\R^N_+)}^q\, du \right)^{1/q},
\end{align*}
which implies \eqref{est:V1}.

By the same idea, we verify
\begin{equation}\label{est:V3}
	\CR_{\CL(L_q(\R^N_+))}(\{(\tau \pd_\tau)^n (\lambda^{1/2} \CV^3(\lambda)) \mid \lambda \in \C_{+, \lambda_2}\})
	\le C_q r (C^1 (\delta_2 + M_1) + (C^2_{M_2, \eta_2} + C^3_{M_2}) \lambda_2^{-1/2}).
\end{equation}
In fact, since $\lambda^{1/2} \CU^3$ is the linear operator from $H^3_q(\R^N_+) \times H^2_q(\R^N_+)^N \times H^3_q(\R^N_+)$ to $H^1_q(\R^N_+)^N$ by Lemma \ref{lem:FG},
it holds that for $n=0, 1$ and for any $m \in \N$, $\{\lambda_j\}_{j=1}^m \subset \C_{+, \lambda_2}$, and $\{\bG_j\}_{j=1}^m \subset \CX_q^2(\R^N_+)$,
\begin{align*}
	&\int^1_0 \Big \|\sum^m_{j=1} r_j(u) \{(\tau \pd_\tau)^n (\lambda^{1/2} \CV^3(\lambda)) \}|_{\lambda = \lambda_j} \bG_j \Big \|_{L_q(\R^N_+)}^q \, du\\
	&\le C_q \Big{[} \{C^1(\delta_2 + 2M_1)\}^q \Big (\int^1_0 \Big \|\sum^m_{j=1}  r_j(u) \{(\tau \pd_\tau)^n \lambda^{1/2} \nabla^2 \CA_1(\lambda)\}|_{\lambda = \lambda_j} \bG_j\Big \|_{L_q(\R^N_+)}^q\, du\\
	&\quad + \int^1_0 \Big \|\sum^m_{j=1}  r_j(u) \{(\tau \pd_\tau)^n \lambda^{1/2} \nabla \CB_1(\lambda)\}|_{\lambda = \lambda_j} \bG_j \Big \|_{L_q(\R^N_+)}^q\, du\Big)\\
	&\quad + (C^1 M_1)^q \int^1_0 \Big \|\sum^m_{j=1}  r_j(u) \{(\tau \pd_\tau)^n \lambda^{1/2} \nabla^2 \CC_1(\lambda)\}|_{\lambda = \lambda_j} \bG_j \Big \|_{L_q(\R^N_+)}^q\, du \\
	&\quad + (C^2_{M_2, \eta_2} + 2C^3_{M_2})^q \Big(\int^1_0 \Big \|\sum^m_{j=1}  r_j(u) \{(\tau \pd_\tau)^n \lambda^{1/2} \nabla \CA_1(\lambda)\}|_{\lambda = \lambda_j} \bG_j \Big \|_{L_q(\R^N_+)}^q\, du \\
	&\quad +\int^1_0 \Big \|\sum^m_{j=1}  r_j(u) \{(\tau \pd_\tau)^n \lambda^{1/2} \nabla \CC_1(\lambda)\}|_{\lambda = \lambda_j} \bG_j \Big \|_{L_q(\R^N_+)}^q\, du \Big)\Big{]} \\
	&=: I_n^3.
\end{align*}
Repeating the same calculation as $I_0^1$, we see that $I_0^3$ satisfies 
\begin{align*}
	&(I_0^3)^{1/q}\\
	&\enskip \le C_q r \{C^1(\delta_2 + 2M_1) + C^1 M_1 K_q^2 \lambda_2^{-1/2} + (C^2_{M_2, \eta_2} + 2C^3_{M_2}) K_q^2 \lambda_2^{-1/2}\}
	\left(\int^1_0 \Big \|\sum^m_{j=1}  r_j(u) \bG_j \Big \|_{\CX_q^2(\R^N_+)}^q\, du\right)^{1/q}\\
	&\enskip \le C_q r \{C^1(\delta_2 + M_1) + (C^2_{M_2, \eta_2} + C^3_{M_2}) \lambda_2^{-1/2}\} \left(\int^1_0 \Big \|\sum^m_{j=1}  r_j(u) \bG_j \Big \|_{\CX_q^2(\R^N_+)}^q\, du \right)^{1/q}.
\end{align*}
In the case $n=1$, note that
\[
	(\tau \pd_\tau) (\lambda^{1/2} \CD(\lambda)) = \frac{1}{\lambda^{1/2}}\left\{(\tau \pd_\tau) (\lambda \CD(\lambda)) - \frac{i}{2}\frac{\tau}{\lambda} (\lambda \CD(\lambda)) \right\},
\] 
where $\CD(\lambda)$ is $\nabla \CA_1(\lambda)$ or $\nabla^j \CC_1(\lambda)$ for $j=1, 2$,
then $(I_1^3)^{1/q}$ also satisfies the same estimate as $(I_0^3)^{1/q}$,
which implies that \eqref{est:V3}.
The other inequalities can be obtained by the same method as $\nabla \CV^1(\lambda)$ and $\lambda^{1/2}\CV^3(\lambda)$, then we have \eqref{Rbound V}.
According to Corollary \ref{apriori} and \eqref{rbdd}, we have
\[
	\|\CF_\lambda^2 \CV(\lambda) \bG\|_{\CX_q^2(\R^N_+)}
	\le 11 r\{C^1(\delta_2 + M_1) + (C^2_{M_2, \eta_2} + C^3_{M_2}) \lambda_2^{-1/2}\}\|\bG\|_{\CX_q^2(\R^N_+)},
\]
which implies that $C_q \ge 11$.
This completes the proof of Lemma \ref{lem:Rbound V}.

\subsection{Proof of Theorem \ref{thm:main bent}}
Let us come back to the resolvent problem in the bent-half space.
Define operators $\CH_j$ ($j=1, 2, 3, 4$) as
\begin{align*}
 \CH_1 \CF_\lambda^2 \widetilde \bF &= \CF_\lambda^2 
	(\widetilde d \circ \Phi, \bA_{-1} \widetilde \bff \circ \Phi, 
	|\bA_\Phi \bn_0| \bA_{-1}^\mathsf{T} \widetilde \bg \circ \Phi, 
	\widetilde k \circ \Phi, \widetilde \zeta \circ \Phi),\\
	\CH_2 \rho &= \rho \circ \Phi^{-1}, \enskip \CH_3 \bv = \bA_{-1}^\mathsf{T} \bv \circ \Phi^{-1}, \enskip \CH_4 h = h \circ \Phi^{-1}.
\end{align*}
Note that \eqref{assumption b}, then we observe that these operators satisfies
\begin{equation}\label{op H}
	\CH_1 \in \CL(\CX_q^2(\Omega_+), \CX_q^2(\R^N_+)), \enskip \CH_2, \CH_4 \in \CL(H^3_q(\R^N_+), H^3_q(\Omega_+)), \enskip \CH_3 \in \CL(H^2_q(\R^N_+)^N, H^2_q(\Omega_+)^N).
\end{equation}
In fact, we verify $\CH_1 \in \CL(\CX_q^2(\Omega_+), \CX_q^2(\R^N_+))$ for $\CF_\lambda^2 \widetilde \bF \in \CX_q^2(\Omega_+)$. 
Here we only consider the estimate of $\nabla_x^2 (\widetilde k \circ \Phi) = \CD_1(\nabla \widetilde k) \circ \Phi + \CD_2(\nabla^2 \widetilde k) \circ \Phi$ since the other terms can be proved by similar way,
where $\CD_1(\nabla \widetilde k) \circ \Phi$ and $\CD_2(\nabla^2 \widetilde k) \circ \Phi$ are $N \times N$ matrices whose $(i, j)^{\rm th}$ components $(\CD_1(\nabla \widetilde k) \circ \Phi)_{ij}$ and $(\CD_2(\nabla^2 \widetilde k) \circ \Phi)_{ij}$ are given by
\begin{align*}
	(\CD_1(\nabla \widetilde k) \circ \Phi)_{ij} &= \sum^N_{k=1} \frac{\pd \widetilde k}{\pd y_k}(\Phi(x)) \frac{\pd^2 \Phi_k}{\pd x_i \pd x_j}\\
	(\CD_2(\nabla^2 \widetilde k) \circ \Phi)_{ij} &= \sum^N_{k, \ell =1} \frac{\pd^2 \widetilde k}{\pd y_k \pd y_\ell}(\Phi(x)) \frac{\pd \Phi_k}{\pd x_i} \frac{\pd \Phi_\ell}{\pd x_j}.
\end{align*}
Then using \eqref{assumption b}, we see that
\[
	\|\nabla_x^2 (\widetilde k \circ \Phi) \|_{L_q(\R^N_+)} 
	\le C^1 M_1^2 \|\nabla_y^2 \widetilde k\|_{L_q(\Omega_+)} + C^3_{M_2}\lambda_2^{-1/2} \|\lambda^{1/2}\nabla_y \widetilde k\|_{L_q(\Omega_+)}
\]
for any $\lambda \in \C_{+, \lambda_2}$, where $C^1$ depends on $N$ and $q$ and $C^3_{M_2}$ depends on $M_2$, $N$, and $q$.
Similarly, the other operators $\CH_2$, $\CH_3$, and $\CH_4$ also satisfy \eqref{op H}.
Set
\begin{align*}
	\CA_2(\lambda) \widetilde \bG &= \CH_2 \CA_3(\lambda) \CH_1 \widetilde \bG,\\
	\CB_2(\lambda) \widetilde \bG &= \CH_3 \CB_3(\lambda) \CH_1 \widetilde \bG,\\
	\CC_2(\lambda) \widetilde \bG &= \CH_4 \CC_3(\lambda) \CH_1 \widetilde \bG
\end{align*}
for $\widetilde \bG \in \CX_q^2(\Omega_+)$.
Thanks to \eqref{rbdd half} and \eqref{op H},
we observe that $\CA_2(\lambda)$, $\CB_2(\lambda)$, and $\CC_2(\lambda)$ are the solution operators for \eqref{r1} and satisfies \eqref{rbdd bent}.
The uniqueness of solutions of \eqref{r1} follows from the uniqueness of solutions of \eqref{r3} proved in the last part of Section \ref{sec:rbdd half}. 
This completes the proof of Theorem \ref{thm:main bent}.

\section{Problem in the general domain}\label{sec.general1}
In this section, we prove the following theorem.
\begin{thm}\label{thm:main general}
Let $1<q<\infty$, and let $\Omega$ be a uniform $C^3$-domain. 
Assume that $\gamma_j$ $(j =1, 2, 3, 4)$ satisfies Assumption \ref{assumption gamma}.
Then there exists a constant $\lambda_3 \ge 1$ such that the following assertions hold true:

\begin{enumerate}
\item
For any $\lambda \in \C_{+, \lambda_3}$ there exist operators  
\begin{align*}
&\CA_4 (\lambda) \in 
{\rm Hol} (\C_{+, \lambda_3}, 
\CL(\CX^2_q(\Omega), H^3_q(\Omega)))\\
&\CB_4 (\lambda) \in 
{\rm Hol} (\C_{+, \lambda_3}, 
\CL(\CX^2_q(\Omega), H^2_q(\Omega)^N)),\\
&\CC_4 (\lambda) \in 
{\rm Hol} (\C_{+, \lambda_3}, 
\CL(\CX^2_q(\Omega), H^3_q(\Omega)))
\end{align*}
such that 
for any $\bF=(d, \bff, \bg, k, \zeta) \in X_q^0(\Omega)$, 
\begin{equation*}
\rho = \CA_4 (\lambda) \CF^2_\lambda \bF, \quad
\bu = \CB_4 (\lambda) \CF^2_\lambda \bF, \quad
h = \CC_4 (\lambda) \CF^2_\lambda \bF
\end{equation*}
are solutions of problem \eqref{main prob}.

\item
There exists a positive constant $r$ such that
\begin{equation} \label{rbdd general}
\begin{aligned}
&\CR_{\CL(\CX^2_q(\Omega), \fA_q^0 (\Omega))}
(\{(\tau \pd_\tau)^n \CR^0_\lambda \CA_4 (\lambda) \mid 
\lambda \in \C_{+, \lambda_3}\}) 
\leq r,\\
&\CR_{\CL(\CX^2_q(\Omega), \fB_q(\Omega))}
(\{(\tau \pd_\tau)^n \CS_\lambda \CB_4 (\lambda) \mid 
\lambda \in \C_{+, \lambda_3}\}) 
\leq r,\\
&\CR_{\CL(\CX^2_q(\Omega), \fC_q(\Omega))}
(\{(\tau \pd_\tau)^n \CT_\lambda \CC_4 (\lambda) \mid 
\lambda \in \C_{+, \lambda_3}\}) 
\leq r
\end{aligned}
\end{equation}
for $n = 0, 1$.
Here, above constant $r$ depend solely on $N$, $q$, $L$, $\gamma_*$ and $\gamma^*$. 
\end{enumerate}
\end{thm} 
\subsection{Preliminaries}
First, we state several properties of uniform $C^3$-domain (cf. \cite[Proposition 6.1]{ES}).
\begin{prop}\label{general domain}
Let $\Omega$ be a uniform $C^3$-domain in $\R^N$ and $M_1$ be the number given in Theorem \ref{thm:main bent}. 
Then there exist four constants $M_2$, $M_3 \ge 1$, $d^1$, $d^2 \in (0, 1)$, depending on $M_1$ at most countably
many $N$-vectors of $C^1$ functions of $\Phi^i_j~(i=1, 2, j \in \N)$ defined on $\BR^N$, and points $x_j^1\in \Omega$, $x_j^2\in \Gamma~(j\in \BN)$ such that, for $i = 1, 2$ and $j \in \N$, the following assertions hold true:
\begin{enumerate}
\item 
The maps $\Phi_j^i : \R^N \ni x \mapsto \Phi_j^i(x) \in \R^N$ are bijective so that $\nabla \Phi_j^i = \bA^i_j + \bB^i_j (x)$ 
and $\nabla (\Phi_j^i)^{-1} = \bA^i_{j,-1} + \bB^i_{j,-1} (x)$ where $\bA^i_j$, $\bA^i_{j,-1}$ are $N\times N$ constant orthonormal matrices
and $\bB^i_j (x)$ and $\bB^i_{j,-1} (x)$ are $N\times N$ matrices of $C^2$ functions defined on $\R^N$, which satisfy
\begin{align*}
	\|(\bB^i_j, \bB^i_{j,-1})\|_{L_{\infty}(\BR^N)}\leq M_1,\quad \|(\nabla\bB^i_j, \nabla\bB^i_{j,-1})\|_{H^1_{\infty}(\R^N)}\leq M_2.
\end{align*}
\item
$\Omega = \left(\bigcup_{j=1}^\infty B_{d^{1}}(x_j^1) \right) \cup \left(\bigcup_{j=1}^\infty \left(\Phi_j^2(\R_+^N)\cap B_{d^{2}}(x_j^2)\right)\right)$,
$B_{d^{1}}(x^1_j) \subset \Omega$, and 
\begin{align*}
	\Phi_j^2(\R_+^N)\cap B_{d^{2}}(x_j^2) &= \Omega \cap B_{d^{2}}(x_j^2)\quad (j\in \N),\\
	\Phi_j^2(\R_0^N)\cap B_{d^{2}}(x_j^2) &= \Gamma \cap B_{d^{2}}(x_j^2)\quad (j\in \N).
\end{align*}
\item \label{g-3}
There exist $C^\infty$ functions $\xi^i_j$, $\widetilde \xi^i_j$ ~$(i=1, 2, ~j \in \N)$ defined on $\R^N$ such that 
\begin{align*}
	& 0\le \xi^i_j, \widetilde \xi^i_j \le 1, \quad {\rm supp}~\xi^i_j, {\rm supp}~\widetilde \xi^i_j
\subset B_{d^1}(x^i_j), \quad \|(\nabla \xi^i_j, \nabla \widetilde \xi^i_j)\|_{H^2_\infty(\R^N)} \le M_3,\\
	&\widetilde \xi^i_j =1 \text{~on~} {\rm supp}~\xi^i_j, \quad \sum^2_{i=1}\sum^\infty_{j=1} \xi^i_j = 1 \text{~on~} \overline{\Omega}, \quad \sum^\infty_{j=1} \xi^2_j = 1 \text{~on~} \Gamma.
\end{align*}
\item \label{g-4}
There exists a natural number $N_0 \ge 2$ such that any $N_0 + 1$ distinct sets of $\{B_{d^{i}}(x_j^i)\mid i = 1, 2, j \in \N\}$. have an empty intersection.

\end{enumerate}
\end{prop}
Hereafter, we denote $B^i_j = B_{d^i}(x^i_j)$ for simplicity.
Let us give some comments on Proposition \ref{general domain}.
\begin{remark}
\begin{enumerate}
\item
Set $\gamma \in \{\gamma_1, \gamma_2, \gamma_3, \gamma_4\}$.
\begin{equation}\label{lip}
\begin{aligned}
	&|\gamma (x) - \gamma(x^1_j)| \le \min \{\delta_1, \delta_2\} \quad \text{ for any } x \in B^1_j,\\
	&|\gamma (x) - \gamma(x^2_j)| \le \min \{\delta_1, \delta_2\} \quad \text{ for any } x \in \overline{\Omega} \cap B^2_j,
\end{aligned}
\end{equation}
where $\delta_1$ and $\delta_2$ are constants given in Proposition \ref{prop:w} and Proposition \ref{main'}, respectively.
\begin{equation}\label{est:normal}
	\|\bn_j\|_{L_\infty(\R^N)} \le 1, \quad \|\nabla \bn_j\|_{H^1_\infty(\R^N)} \le C_{M_2}
\end{equation}
for some positive constant $C_{M_2}$ independent of $j$.
\item
Let $\bn_j$ be the unit outer normal to $\Gamma_j$, and let $\Delta_{\Gamma_j}$ be the Laplace-Beltrami operator on $\Gamma_j$.
Note that
\begin{equation}\label{normal}
\bn = \bn_j \quad \Delta_{\Gamma} = \Delta_{\Gamma_j} \quad \text{ on } \Gamma \cap B^2_j
\end{equation}
for $j \in \N$.
\item
\begin{equation}\label{intersect}
	\left(\sum^2_{i=1}\sum^\infty_{j=1} \|f\|_{L_r(\Omega \cap B^i_j)}^r\right)^{1/r} \le M_4 \|f\|_{L_r(\Omega)}
\end{equation}
for any $f \in L_r(\Omega)$.
\end{enumerate}
\end{remark}

Next, we prepare some lemmas used to construct a parametrix.
The first lemma was proved in \cite[Proposition 9.5.2]{S2016} and the second lemma was introduced in \cite[Proposition 5.2 (ii)]{S2014}.
\begin{lem}\label{lem:p1}
Let $X$ be a Banach space and $X^*$ its dual space, while $\|\cdot\|_X$, $\|\cdot\|_{X^*}$, and $\langle \cdot, \cdot \rangle$ be the norm of $X$, the norm of $X^*$,
and the duality pairing between $X$ and $X^*$, respectively.
Let $m \in \N$, $\ell = 1, \dots, m$, and $\{a_\ell\}_{\ell=1}^m \subset \C$, and let $\{f^\ell_j\}_{j=1}^\infty$ be sequences in $X^*$ and $\{g^\ell_j\}_{j=1}^\infty$, $\{h_j\}_{j=1}^\infty$
be sequences of positive numbers.
Assume that there exist map $\CN_j : X \to [0, \infty)$ such that
\[
	|\langle f^\ell_j. \vp \rangle| \le M_5 g^\ell_j \CN_j(\vp)~(\ell = 1, \dots, m), 
	\quad \left|\langle \sum^m_{\ell=1} a_\ell f^\ell_j, \vp \rangle \right| \le M_5 h_j \CN_j(\vp)
\]
for any $\vp \in X$ with some positive constant $M_5$ independent of $j \in \N$ and $\ell = 1, \dots, m$. If
\[
	\sum^\infty_{j=1} (g^\ell_j)^q < \infty, \quad \sum^\infty_{j=1} (h_j)^q < \infty, \quad \sum^\infty_{j=1} (\CN_j(\vp))^{q'} \le (M_6 \|\vp\|_X)^{q'}
\]
with $1<q<\infty$ and $q'=q/(q-1)$ for some positive constant $M_6$, then infinite sum $f^\ell = \sum^\infty_{j=1} f^\ell_j$ exists in the strong topology of $X^*$ and 
\[
	\|f^\ell\|_{X^*} \le M_5 M_6 \left(\sum^\infty_{j=1} (g^\ell_j)^q\right)^{1/q},\quad
	\left\|\sum^m_{\ell=1}a_\ell f^\ell \right\|_{X^*} \le M_5 M_6 \left(\sum^\infty_{j=1} (h_j)^q\right)^{1/q}.
\]
\end{lem}

\begin{lem}\label{lem:p2}
Let $1< q < \infty$, $q'=q/(q-1)$, and $i=1, 2$. 
Let $m \in \N_0$, $\{f_j\}_{j=1}^\infty$ be a sequence in $H^m_q(\Omega)$, and let $\{g^\ell_j\}_{j=1}^\infty$
be a sequence of positive numbers.
Assume that 
\[
	\sum^\infty_{j=1} (g^\ell_j)^q < \infty, \quad |(\nabla^\ell f_j, \vp)_\Omega| \le M_7 g^\ell_j \|\vp\|_{L_{q'}(\Omega \cap B^i_j)}
\]
for any $\vp \in L_{q'}(\Omega)$ with some positive constant $M_7$ independent of $j\in \N$ and $\ell = 0, 1, \dots, m$.
Then infinite sum $f = \sum^\infty_{j=1} f_j$ exists in the strong topology of $H^m_q(\Omega)$ and
\[
	\|\nabla^\ell f\|_{L_q(\Omega)} \le C_{q, M_4} M_7 \left(\sum^\infty_{j=1} (g^\ell_j)^q\right)^{1/q}
\]
with some positive constant $C_{q, M_4}$.
\end{lem}

\subsection{Localization}\label{local}
Let $\Omega^1_j = \R^N$, $\Omega^2_j = \Phi^2_j (\R^N_+)$, and $\Gamma_j = \Phi^2_j (\R^N_0)$ for $j \in \N$,
and let
\begin{equation}\label{local gamma}
	\gamma^i_{kj} (x) = (\gamma_k(x) - \gamma_k(x^i_j)) \widetilde \xi^i_j(x) + \gamma_k(x^i_j)
\end{equation}
for $x \in \Omega^i_j$, $i = 1, 2$, $j \in \N$, $k = 1, 2, 3, 4$.
Then it holds by \eqref{lip} that
\begin{equation}\label{lip1}
\begin{aligned}
	&\sup_{x \in \Omega^1_j} |\gamma^1_{kj}(x) - \gamma_k(x^1_j)| \le \sup_{x \in B^1_j}|\gamma (x) - \gamma(x^1_j)| \le \min \{\delta_1, \delta_2\},\\
	&\sup_{x \in \Omega^2_j} |\gamma^2_{kj}(x) - \gamma_k(x^2_j)| \le \sup_{x \in \overline{\Omega} \cap B^2_j}|\gamma (x) - \gamma(x^2_j)| \le \min \{\delta_1, \delta_2\}
\end{aligned}
\end{equation}
for $k = 1, 2, 3, 4$ and $j \in \N$, then together with Remark \ref{rem:lip}, 
\begin{equation}\label{lip2}
	\|\nabla \gamma^i_{kj}\|_{L_\infty(\Omega^i_j)} \le L + \gamma^* + M_3
\end{equation}
for $i =1, 2$, $k = 1, 2, 3$, and $j \in \N$.
Assumption \ref{assumption gamma} provide that
\begin{equation}\label{assumption bent}
	(\gamma_2(x^2_j) + \gamma_3(x^2_j))^2 \neq 2(\gamma_1(x^2_j))^2\gamma_4(x^2_j), \quad (\gamma_1(x^2_j))^2 \gamma_4(x^2_j) \neq \gamma_2(x^2_j) \gamma_3(x^2_j).
\end{equation}

Set 
\[
	\bS^i_j(\bu) = \gamma^i_{2j} \bD(\bu) + (\gamma^i_{3j} - \gamma^i_{2j}) \dv \bu \bI
\]
for $i = 1, 2$ and $j \in \N$, then
we consider the following problems:
\begin{equation}\label{local w}
\left\{
\begin{aligned}
	&\lambda \rho^1_j + \gamma^1_{1j} \dv \bu^1_j = \widetilde \xi^1_j d & \quad&\text{in $\Omega^1_j$}, \\
	&\lambda \bu^1_j - \gamma_4^{-1} 
	\DV(\bS^1_j(\bu^1_j) + \gamma^2_{1j} \Delta \rho^1_j \bI)=  \widetilde \xi^1_j \bff& \quad&\text{in $\Omega^1_j$},
\end{aligned}
\right.
\end{equation}
\begin{equation}\label{local b}
\left\{
\begin{aligned}
	&\lambda \rho^2_j + \gamma^2_{1j} \dv \bu^2_j = \widetilde \xi^2_j d & \quad&\text{in $\Omega^2_j$}, \\
	&\lambda \bu^2_j - \gamma_4^{-1} 
	\DV(\bS^2_j(\bu^2_j) + \gamma^2_{1j} \Delta \rho^2_j \bI)=  \widetilde \xi^2_j \bff& \quad&\text{in $\Omega^2_j$},\\
	&(\bS^2_j(\bu^2_j) + \gamma^2_{1j} \Delta \rho^2_j \bI) \bn_j - \sigma \Delta_{\Gamma_j} h_j \bn_j =  \widetilde \xi^2_j \bg & \quad&\text{on $\Gamma_j$},\\
	&\bn_j \cdot \nabla \rho^2_j =  \widetilde \xi^2_j k & \quad&\text{on $\Gamma_j$},\\
	&\lambda h_j - \bu^2_j \cdot \bn_j = \widetilde \xi^2_j \zeta & \quad&\text{on $\Gamma_j$}.
\end{aligned}
\right.
\end{equation}
Thanks to \eqref{lip1}, \eqref{lip2}, and \eqref{assumption bent}, we can apply Proposition \ref{prop:w} and Theorem \ref{thm:main bent} to problems \eqref{local w} and \eqref{local b}, respectively, then
we observe that there exists a constant $\omega_0$ such that the following assertions hold:

$\thetag1$
For any $\lambda \in \C_{+, \omega_0}$ there exist operators  
\begin{align*}
	&\CA^i_j (\lambda) \in 
	{\rm Hol} (\C_{+, \omega_0}, 
	\CL(\CX^i_q(\Omega^i_j), H^3_q(\Omega^i_j)))\\
	&\CB^i_j (\lambda) \in 
	{\rm Hol} (\C_{+, \omega_0}, 
	\CL(\CX^i_q(\Omega^i_j), H^2_q(\Omega^i_j)^N)),\\
	&\CC_j (\lambda) \in 
	{\rm Hol} (\C_{+, \omega_0}, 
	\CL(\CX^2_q(\Omega^2_j), H^3_q(\Omega^2_j)))
\end{align*}
such that 
for any $\bF=(d, \bff, \bg, k, \zeta) \in X_q^0(\Omega)$, 
\begin{equation}\label{localized sol}
\begin{aligned}
	(\rho^1_j, \bu^1_j) &= (\CA^1_j (\lambda) \CF^1_\lambda \bF^1, \CB^1_j (\lambda) \CF^1_\lambda \bF^1),\\
	(\rho^2_j, \bu^2_j, h_j) &= (\CA^2_j (\lambda) \CF^2_\lambda \bF^2, \CB^2_j (\lambda) \CF^2_\lambda \bF^2, \CC_j (\lambda) \CF^2_\lambda \bF^2),
\end{aligned}
\end{equation}
are a unique solutions of problems \eqref{local w} and \eqref{local b}, respectively,
where we have set $\bF^1 = (\widetilde \xi^1_j d, \widetilde \xi^1_j \bff)$ and 
$\bF^2 = (\widetilde \xi^2_j d, \widetilde \xi^2_j \bff, \widetilde \xi^2_j \bg, \widetilde \xi^2_j k, \widetilde \xi^2_j \zeta)$.

$\thetag2$
There exists a positive constant $r$ such that
\begin{equation} \label{local rbdd}
\begin{aligned}
	&\CR_{\CL(\CX^i_q(\Omega^i_j), \fA_q^0 (\Omega^i_j))}
	(\{(\tau \pd_\tau)^n \CR_\lambda^0 \CA^i_j (\lambda) \mid 
	\lambda \in \C_{+, \omega_0}\}) 
	\leq r,\\
	&\CR_{\CL(\CX^i_q(\Omega^i_j), \fB_q(\Omega^i_j))}
	(\{(\tau \pd_\tau)^n \CS_\lambda \CB^i_j (\lambda) \mid 
	\lambda \in \C_{+, \omega_0}\}) 
	\leq r,\\
	&\CR_{\CL(\CX^2_q(\Omega^2_j), \fC_q(\Omega^2_j))}
	(\{(\tau \pd_\tau)^n \CT_\lambda \CC_j (\lambda) \mid 
	\lambda \in \C_{+, \omega_0}\}) 
	\leq r
\end{aligned}
\end{equation}
for $n = 0, 1$, $i = 1, 2$, $j \in \N$.
Here, above constant $r$ depend solely on $M_3$, $N$, $q$, $L$, $\gamma_*$ and $\gamma^*$, but independent of $j \in \N$. 
In particular,
\begin{equation}\label{local resolvent}
\begin{aligned}
	&\|(\CR_\lambda^0 \rho^1_j, \CS_\lambda \bu^1_j)\|_{\fA_q^0 (\Omega^1_j) \times \fB_q (\Omega^1_j)} \le C_{M_3} r \|(\nabla d, \lambda^{1/2}d, \bff)\|_{L_q(\Omega \cap B^1_j)},\\
	&\|(\CR_\lambda^0 \rho^2_j, \CS_\lambda \bu^2_j, \CT_\lambda h_j)\|_{\fA_q^0 (\Omega^2_j) \times \fB_q (\Omega^2_j) \times \fC_q (\Omega^2_j)} \\
	&\quad \le C_{M_3} r 
	(\|(\nabla d, \lambda^{1/2} d, \bff, \nabla \bg, \lambda^{1/2} \bg, \nabla^2 k, \nabla \lambda^{1/2}k, \lambda k)\|_{L_q(\Omega \cap B^2_j)}+\|\zeta\|_{H^2_q(\Omega \cap B^2_j)}),
\end{aligned}
\end{equation}
where we have used Proposition \ref{general domain} \eqref{g-3}.

\subsection{Construction of a parametrix}
According to \eqref{intersect} and \eqref{local resolvent}, 
$\sum^2_{i=1} \sum^\infty_{j=1} \xi^i_j \rho^i_j$, $\sum^2_{i=1} \sum^\infty_{j=1} \xi^i_j \bu^i_j$, and $\sum^\infty_{j=1} \xi^2_j h_j$ exist in the strong topology of $H_q^3(\Omega)$, $H_q^2(\Omega)^N$, and $H_q^3(\Omega)$ for $(d, \bff, \bg, k, \zeta) \in X^0_q(\Omega)$, respectively.
Thus, let us define 
\begin{equation}\label{def sol}
	\rho = \sum^2_{i=1} \sum^\infty_{j=1} \xi^i_j \rho^i_j \text{ in~$H^3_q(\Omega)$}, \quad
	\bu = \sum^2_{i=1} \sum^\infty_{j=1} \xi^i_j \bu^i_j \text{ in~$H^2_q(\Omega)^N$}, \quad
	h = \sum^\infty_{j=1} \xi^2_j h_j \text{ in~$H^3_q(\Omega)$}.
\end{equation}
Furthermore, \eqref{local w}, \eqref{local b}, Proposition \ref{general domain}, and \eqref{normal} imply that $\rho$, $\bu$, and $h$ satisfy the following problem:

\begin{equation}\label{p-prob}
\left\{
\begin{aligned}
	&\lambda \rho + \gamma_1 \dv \bu = d - \CU_1(\lambda)\bF & \quad&\text{in $\Omega$}, \\
	&\lambda \bu - \gamma_4^{-1} 
	\DV(\bS(\bu) + \gamma_1  \Delta \rho \bI) = \bff - \CU_2(\lambda)\bF& \quad&\text{in $\Omega$},\\
	&(\bS(\bu) + \gamma_1  \Delta \rho \bI) \bn - \sigma \Delta_{\Gamma} h \bn =\bg - \CU_3(\lambda)\bF & \quad&\text{on $\Gamma$},\\
	&\bn \cdot \nabla \rho = k - \CU_4(\lambda)\bF& \quad&\text{on $\Gamma$},\\
	&\lambda h - \bu \cdot \bn = \zeta & \quad&\text{on $\Gamma$},
\end{aligned}
\right.
\end{equation}
where $\bS(\bu) = \gamma_2 \bD(\bu) + (\gamma_3-\gamma_2) \dv \bu \bI$ and
\begin{equation}\label{remainder}
\begin{aligned}
	\CU_1(\lambda) \bF &= \sum^2_{i=1} \sum^\infty_{j=1}(\xi^i_j \gamma^i_{1j} \dv \bu^i_j - \gamma_1 \dv(\xi^i_j \bu^i_j)),\\ 
	\CU_2(\lambda) \bF &= -\sum^2_{i=1} \sum^\infty_{j=1}\xi^i_j(\gamma^i_{4j})^{-1}\DV(\bS^i_j(\bu^i_j) + \gamma^i_{1j} \Delta \rho^i_j \bI)\\
	&\enskip + \sum^2_{i=1} \sum^\infty_{j=1} \gamma_4^{-1} \DV(\bS(\xi^i_j \bu^i_j) + \gamma_1 \Delta(\xi^i_j \rho^i_j)\bI),\\
	\CU_3(\lambda) \bF &= \sum^\infty_{j=1}\xi^2_j (\bS^2_j(\bu^2_j) \bn_j + \gamma^2_{1j} \Delta \rho^2_j \bn_j) - \sigma \sum^\infty_{j=1} \xi^2_j \Delta_{\Gamma_j} h_j \bn_j\\
		&\enskip - \sum^\infty_{j=1} (\bS(\xi^2_j \bu^2_j) + \gamma_1 \Delta(\xi^2_j \rho^2_j)\bI)\bn + \sigma \sum^\infty_{j=1} \Delta_{\Gamma} (\xi^2_j h_j) \bn,\\
	\CU_4(\lambda) \bF &= \sum^\infty_{j=1} (\xi^2_j \bn_j \cdot \nabla \rho^2_j - \bn \cdot \nabla(\xi^2_j \rho^2_j)).
\end{aligned}
\end{equation}
Let $\bH=(H_1, \dots, H_9) \in \CX^2_q(\Omega)$, where $H_1, \dots, H_9$ are variables corresponding to
$\nabla d$, $\lambda^{1/2}d$, $\bff$, $\nabla \bg$, $\lambda^{1/2}\bg$, $\nabla^2 k$, $\nabla \lambda^{1/2} k$, $\lambda k$, and $\zeta$, respectively.
Then the right-hand side of the following identity
\[
	\CF^1_\lambda \bF^1 = (\nabla (\widetilde \xi^1_j d), \lambda^{1/2}\widetilde \xi^1_j d, \widetilde \xi^1_j \bff) 
	= ( \widetilde \xi^1_j \nabla d + \lambda^{-1/2}( \nabla \widetilde \xi^1_j) \lambda^{1/2} d, \widetilde \xi^1_j \lambda^{1/2} d, \widetilde \xi^1_j \bff)
\]
corresponds to $( \widetilde \xi^1_j H_1 + \lambda^{-1/2}( \nabla \widetilde \xi^1_j) H_2, \widetilde \xi^1_j H_2, \widetilde \xi^1_j H_3)$,
where $\bF^1$ was defined in subsection \ref{local}.
Similarly, we can consider for $\CF^2_\lambda \bF^2$.
In view of these correspondence and \eqref{localized sol}, we define for $(Z^1_j, \CZ^1_j) \in \{(A^1_j, \CA^1_j), (B^1_j, \CB^1_j)\}$ and for $(Z^2_j, \CZ^2_j) \in \{(A^2_j, \CA^2_j), (B^2_j, \CB^2_j), (C_j, \CC_j)\}$
\begin{align*}
	Z^1_j(\lambda) \bH &= \CZ^1_j(\lambda)(\widetilde \xi^1_j H_1 + \lambda^{-1/2}( \nabla \widetilde \xi^1_j) H_2, \widetilde \xi^1_j H_2, \widetilde \xi^1_j H_3),\\
	Z^2_j(\lambda) \bH &= \CZ^2_j(\lambda)(\widetilde \xi^2_j H_1 + \lambda^{-1/2}( \nabla \widetilde \xi^2_j) H_2, \widetilde \xi^2_j H_2, \widetilde \xi^2_j H_3, \lambda^{-1/2}(\nabla \widetilde \xi^2_j) \otimes H_5 + \widetilde \xi^2_j H_4, \widetilde \xi^2_j H_5, \\
	&\enskip \widetilde \xi^2_j H_6 + \lambda^{-1/2} \nabla \widetilde \xi^2_j \otimes H_7 + \lambda^{-1/2}H_7 \otimes \nabla \widetilde \xi^2_j + \lambda^{-1} (\nabla^2\widetilde \xi^2_j) H_8,\\
&\enskip \widetilde \xi^2_j H_7 + \lambda^{-1/2}(\nabla \widetilde \xi^2_j)H_8, \widetilde \xi^2_j H_8, \widetilde \xi^2_j H_9).
\end{align*}

Here we only consider the operator $C_j(\lambda)$ for instance.
Let $\lambda_3$ be a arbitrary number with $\lambda_3 \ge \omega_0$.
It holds by \eqref{local rbdd} and Proposition \ref{general domain} that for any $m \in \N$, $\{\lambda_\ell\}_{\ell=1}^m \subset \C_{+, \lambda_3}$, $\{\bH_\ell\}_{\ell=1}^m \subset \CX^2_q(\Omega)$, and $k=0, 1$
\begin{equation}\label{rbdd c1}
\begin{aligned}
	&\left(\int^1_0 \left \|\sum^m_{\ell=1} r_\ell(u) \lambda_\ell^k \xi^2_j C_j(\lambda_\ell) \bH_\ell \right \|_{H^{3-k}_q(\Omega^2_j)}^q \, du\right)^{1/q}\\
	&\le r C_{q, \omega_0, M_3} \left(\int^1_0 \left\|\sum^m_{\ell=1}  r_\ell(u) \bH_\ell \right \|_{\CX^2_q(\Omega \cap B^2_j)}^q\, du \right)^{1/q}.
\end{aligned}
\end{equation}
Lemma \ref{lem:p2}, together with \eqref{intersect} and \eqref{rbdd c1} with $m=1$, yields that
$\sum^\infty_{j=1} \xi^2_j C_j(\lambda) \bH$ exists in the strong topology of $H^3_q(\Omega)$, thus we define 
\begin{equation}\label{def c}
	C(\lambda) \bH = \sum^\infty_{j=1} \xi^2_j C_j(\lambda) \bH 
\end{equation}
for $\bH \in \CX^2_q(\Omega)$. 
Furthermore, by similar method as in the proof of \cite[Proposition 5.3]{S2014}, we see that
$C(\lambda) \in {\rm Hol} (\C_{+, \lambda_3}, \CL(\CX^2_q(\Omega), H^3_q(\Omega)))$
by \eqref{intersect} and \eqref{rbdd c1} with $m=1$.
Let us prove the $\CR$-boundedness for $\CT_\lambda C(\lambda)$ by \eqref{rbdd c1}.
Note that for any $m \in \N$, $\{a_\ell\}_{\ell=1}^m \subset \C$, $\{\lambda_\ell\}_{\ell=1}^m \subset \C_{+, \lambda_3}$, and $\{\bH_\ell\}_{\ell=1}^m \subset \CX^2_q(\Omega)$
\begin{align*}
	&\sum^\infty_{j=1} \left( \sum^m_{\ell=1} |a_\ell| \|\lambda_\ell^k (\xi^2_j C_j(\lambda_\ell)\bH_\ell)\|_{H^{3-k}_q(\Omega_j^2)}\right)^q\\
	&\le r^q C_{q, \omega_0, M_3} \sum^m_{\ell=1} |a_\ell|^q \sum^\infty_{j=1} \|\bH_\ell\|_{\CX^2_q(\Omega \cap B^2_j)}^q\\
	&\le r^q C_{q, \omega_0, M_3}M_4^q \sum^m_{\ell=1} |a_\ell|^q \|\bH_\ell\|_{H^{3-k}_q(\Omega)}^q
\end{align*}
by \eqref{rbdd c1} and \eqref{intersect}. 
Then it holds by Lemma \ref{lem:p1}, \eqref{rbdd c1}, \eqref{intersect}, together with the monotone convergence theorem,
\begin{equation}\label{rbdd c2}
\begin{aligned}
	&\int^1_0 \left \|\sum^m_{\ell=1} r_\ell(u) \lambda_\ell^k C(\lambda_\ell) \bH_\ell \right \|_{H^{3-k}_q(\Omega)}^q \, du\\
	&\le r^q C_{q, \omega_0, M_3}M_4^q \sum^\infty_{j=1} \int^1_0 \left \|\sum^m_{\ell=1} r_\ell(u) \lambda^k_\ell C(\lambda_\ell) \bH_\ell \right \|_{H^{3-k}_q(\Omega_j^2)}^q\, du \\
	&\le r^q C_{q, \omega_0, M_3}M_4^q \sum^\infty_{j=1} \int^1_0 \left \|\sum^m_{\ell=1} r_\ell(u) \bH_\ell \right \|_{\CX^2_q(\Omega \cap B^2_j)}^q\, du\\
	&\le r^q C_{q, \omega_0, M_3}M_4^{2q} \int^1_0 \left \|\sum^m_{\ell=1} r_\ell(u) \bH_\ell \right \|_{\CX^2_q(\Omega)}^q\, du.
\end{aligned}
\end{equation}
Therefore we obtain
\[
	\CR_{\CL(\CX^2_q(\Omega), \fC_q(\Omega))}
\left(\{\CT_\lambda C(\lambda) \mid \lambda \in \C_{+, \lambda_3}\}\right)
\le R,
\]
where we have set $R=r (C_{q, \omega_0, M_3})^{1/q}M_4^2$.
Similarly, we can consider the $\CR$-boundedness for $(\tau \pd_\tau)\CT_\lambda C(\lambda)$,
then we reach
\begin{equation}\label{rbdd c}
	\CR_{\CL(\CX^2_q(\Omega), \fC_q(\Omega))}
\left(\{(\tau \pd_\tau)^n \CT_\lambda C(\lambda) \mid \lambda \in \C_{+, \lambda_3}\}\right)
\le R
\end{equation}
for $n=0, 1$.  
Repeating the above calculation, define
\begin{equation}\label{def ab}
	A(\lambda) \bH = \sum^2_{i=1} \sum^\infty_{j=1} \xi^i_j A^i_j(\lambda) \bH, \quad B(\lambda) \bH = \sum^2_{i=1} \sum^\infty_{j=1} \xi^i_j B^i_j(\lambda) \bH
\end{equation}
for $\bH \in \CX^2_q(\Omega)$, then we see that 
$A(\lambda) \in {\rm Hol} (\C_{+, \lambda_3}, \CL(\CX^2_q(\Omega), H^3_q(\Omega)))$,
$B(\lambda) \in {\rm Hol} (\C_{+, \lambda_3}, \CL(\CX^2_q(\Omega), H^2_q(\Omega)^N))$,
and
\begin{equation}\label{rbdd ab}
\begin{aligned}
	\CR_{\CL(\CX^2_q(\Omega), \fA^0_q(\Omega))}
\left(\{(\tau \pd_\tau)^n \CR^0_\lambda A(\lambda) \mid \lambda \in \C_{+, \lambda_3}\}\right)
\le R,\\
	\CR_{\CL(\CX^2_q(\Omega), \fB_q(\Omega))}
\left(\{(\tau \pd_\tau)^n \CS_\lambda B(\lambda) \mid \lambda \in \C_{+, \lambda_3}\}\right)
\le R
\end{aligned}
\end{equation}
for $n=0, 1$.
Furthermore, by \eqref{def c} and \eqref{def ab}, together with \eqref{localized sol} and \eqref{def sol},
the solution $(\rho, \bu, h)$ of \eqref{p-prob}
can be written as
\[
	\rho = A(\lambda)\CF_\lambda^2 \bF, \enskip \bu = B(\lambda) \CF_\lambda^2 \bF, \enskip h = C(\lambda) \CF_\lambda^2 \bF
\]
for $\bF = (d, \bff, \bg, k, \zeta) \in X^0_q(\Omega)$.

\subsection{Proof of Theorem \ref{thm:main general}}
In the same way as subsection \ref{sec:rbdd half}, we can obtain the desired solution operators for \eqref{main prob}.
Set 
\[
	\CU(\lambda) \bF = (\CU_1(\lambda) \bF, \CU_2(\lambda) \bF, \CU_3(\lambda) \bF, \CU_4(\lambda) \bF, 0).
\]
The proof of the following lemma is postponed to the next subsection.
\begin{lem}\label{lem:Rbound V g}
Let $1< q <\infty$, and let $\Omega$ be a uniform $C^3$-domain. Let $\lambda_3$ be the constants mentioned above.
Then for any $\lambda \in \C_{+, \lambda_3}$ there exists an operator 
\[
	\CV(\lambda) \in {\rm Hol} (\C_{+, \lambda_3}, \CL(\CX^2_q(\Omega), X^0_q(\Omega)))
\]
such that $\CU(\lambda) \bF = \CV(\lambda) \CF^2_\lambda \bF$ for $\bF \in X^0_q(\Omega)$ and
\begin{equation}\label{Rbound V g}
	\CR_{\CL(\CX_q^2(\Omega))}\{(\tau \pd_\tau)^n \CF_\lambda^2 \CV(\lambda) \mid \lambda \in \C_{+, \lambda_3}\}
\le C_{q, \omega_0, R, M_2, M_3} \lambda_3^{-1/2}
\end{equation}
for $n=0, 1$, where $C_{q, \omega_0, R, M_2, M_3}$ is some positive constant.
\end{lem}
Admitting Lemma \ref{lem:Rbound V g} and $\lambda_3$ so large that
$
	C \lambda_3^{-1/2} \le 1/2,
$
we observe that there exists the inverse operator $(\bI - \CF_\lambda^2 \CV(\lambda))^{-1}$ satisfying
\begin{equation}\label{Rbound inverse V g}
	\CR_{\CL(\CX_q^2(\Omega))}\{(\tau \pd_\tau)^n (\bI - \CF_\lambda^2 \CV(\lambda))^{-1} \mid \lambda \in \C_{+, \lambda_3}\}
\le 2.
\end{equation}
Set 
\begin{align*}
	\CA_4(\lambda) \bH &= A(\lambda) (\bI - \CF_\lambda^2\CV(\lambda))^{-1} \bH,\\
	\CB_4(\lambda) \bH &= B(\lambda) (\bI - \CF_\lambda^2\CV(\lambda))^{-1} \bH,\\
	\CC_4(\lambda) \bH &= C(\lambda) (\bI - \CF_\lambda^2\CV(\lambda))^{-1} \bH
\end{align*}
for $\bH \in \CX_q^2(\Omega)$,
then
$	
	(\rho, \bu, h) = (\CA_4(\lambda) \bH, \CB_4(\lambda) \bH, \CC_4(\lambda) \bH)
$
solves \eqref{main prob}. In addition,  Lemma \ref{lem:5.3} (\ref{lem:Rbound2}), \eqref{rbdd c}, \eqref{rbdd ab}, and \eqref{Rbound inverse V g} furnish that
\[
\begin{aligned}
	&\CR_{\CL(\CX_q^2(\Omega), \fA_q^0 (\Omega))}
	(\{(\tau \pd_\tau)^n \CR_\lambda^0 \CA_4 (\lambda) \mid 
	\lambda \in \C_{+, \lambda_3}\}) 
	\leq 2R,\\
	&\CR_{\CL(\CX_q^2(\Omega), \fB_q(\Omega))}
	(\{(\tau \pd_\tau)^n \CS_\lambda \CB_4 (\lambda) \mid 
	\lambda \in \C_{+, \lambda_3}\}) 
	\leq 2R,\\
	&\CR_{\CL(\CX_q^2(\Omega), \fC_q(\Omega))}
	(\{(\tau \pd_\tau)^n \CT_\lambda \CC_4 (\lambda) \mid 
	\lambda \in \C_{+, \lambda_3}\}) 
	\leq 2R
\end{aligned}
\]
for $n = 0, 1$.

The uniqueness of the solution to \eqref{main prob} follows from the same method as \cite[subsection 3.3]{S2020}.
Let $1<q<\infty$ and $q'=q/(q-1)$, and let $\lambda_3=\lambda_3(q)$ be the positive number given in the previous subsection.
Assume $\lambda \in \C_{+, \max\{\lambda_3(q), \lambda_3(q')\}}$ and $(\rho, \bu, h) \in H^3_q(\Omega) \times H^2_q(\Omega)^N \times H^3_q(\Omega)$ satisfies
\begin{equation}\label{homo}
\left\{
\begin{aligned}
	&\lambda \rho + \gamma_1 \dv \bu = 0 & \quad&\text{in $\Omega$}, \\
	&\lambda \bu - \gamma_4^{-1} 
	\DV\{\gamma_2 \bD(\bu) + (\gamma_3 - \gamma_2) \dv \bu \bI + \gamma_1  \Delta \rho \bI\}= 0& \quad&\text{in $\Omega$},\\
	&\{\gamma_2 \bD(\bu) + (\gamma_3 - \gamma_2) \dv \bu \bI + \gamma_1  \Delta \rho \bI \} \bn - \sigma \Delta_{\Gamma} h \bn =0 & \quad&\text{on $\Gamma$},\\
	&\bn \cdot \nabla \rho = 0 & \quad&\text{on $\Gamma$},\\
	&\lambda h - \bu \cdot \bn = 0 & \quad&\text{on $\Gamma$}.
\end{aligned}
\right.
\end{equation}
Let $\vp \in C^\infty_0(\Omega)$.
Since $\gamma_4^{-1} \vp \in L_{q'}(\Omega)$, there exists a solution $(\theta, \bv, \eta)
\in
H^3_{q'}(\Omega) \times H^2_{q'}(\Omega)^N \times H^3_{q'}(\Omega)$ satisfies
\begin{equation}\label{dual}
\left\{
\begin{aligned}
	&\lambda \theta + \gamma_1 \dv \bv = 0 & \quad&\text{in $\Omega$}, \\
	&\lambda \bv - \gamma_4^{-1} 
	\DV\{\gamma_2 \bD(\bv) + (\gamma_3 - \gamma_2) \dv \bv \bI + \gamma_1 \Delta \theta \bI\}= \gamma_4^{-1} \vp& \quad&\text{in $\Omega$},\\
	&\{\gamma_2 \bD(\bv) + (\gamma_3 - \gamma_2) \dv \bv \bI + \gamma_1  \Delta \theta \bI \} \bn - \sigma \Delta_{\Gamma} \eta \bn =0 & \quad&\text{on $\Gamma$},\\
	&\bn \cdot \nabla \theta = 0 & \quad&\text{on $\Gamma$},\\
	&\lambda \eta - \bv \cdot \bn = 0 & \quad&\text{on $\Gamma$}.
\end{aligned}
\right.
\end{equation}
Then we have
\[
	(\bu, \vp)_\Omega = (\bu, \gamma_4 \lambda \bv -  
	\DV\{\gamma_2 \bD(\bv) + (\gamma_3 - \gamma_2) \dv \bv \bI + \gamma_1 \Delta \theta \bI\})_\Omega,
\]
together with
\begin{equation}\label{div}
	(\bw, \DV \bM)_\Omega = (\bw, \bM \bn)_\Gamma - \frac12 (\bD(\bw), \bM)_\Omega
\end{equation}
for the vector of function $\bw$ and matrix field $\bM$ with $\bM^\mathsf{T}=\bM$,
furnishes that
\begin{align*}
	(\bu, \vp)_\Omega &= (\gamma_4 \lambda \bu, \bv)_\Omega -  
	(\bu, \{\gamma_2 \bD(\bv) + (\gamma_3 - \gamma_2) \dv \bv \bI + \gamma_1 \Delta \theta \bI\}\bn)_\Gamma \\
	&\enskip + \frac12(\bD(\bu), \gamma_2 \bD(\bv) + (\gamma_3 - \gamma_2) \dv \bv \bI + \gamma_1 \Delta \theta \bI)_\Omega\\
	&= (\gamma_4 \lambda \bu, \bv)_\Omega - (\bu, \sigma \Delta_\Gamma \eta \bn)_\Gamma\\
	&\enskip + \frac12\{(\gamma_2 \bD(\bu) + (\gamma_3 - \gamma_2) \dv \bu \bI, \bD(\bv))_\Omega + 2(\gamma_1 \dv \bu, \Delta \theta)_\Omega\},
\end{align*}
where we have used the third equation of \eqref{dual}.
Here the first equations of \eqref{homo} and \eqref{dual}, combined with the fourth equations of \eqref{homo} and \eqref{dual}, together with the integration by parts, imply that
\begin{align*}
	(\gamma_1 \dv \bu, \Delta \theta)_\Omega 
	&= -(\lambda \rho, \Delta \theta)_\Omega =-(\Delta \rho, \lambda \theta)_\Omega = (\Delta \rho, \gamma_1 \dv \bv)_\Omega\\
	&=(\gamma_1 \Delta \rho, \dv \bv)_\Omega = \frac12 (\gamma_1 \Delta \rho \bI, \bD(\bv))_\Omega. 
\end{align*}
Therefore, by \eqref{div} and the integration by parts,
\begin{align*}
	(\bu, \vp)_\Omega 
	&=(\gamma_4 \lambda \bu, \bv)_\Omega - (\bu, \sigma \Delta_\Gamma \eta \bn)_\Gamma 
	+ \frac12(\gamma_2 \bD(\bu) + (\gamma_3 - \gamma_2) \dv \bu \bI+\gamma_1 \Delta \rho \bI, \bD(\bv))_\Omega \\
	&=(\gamma_4 \lambda \bu, \bv)_\Omega - (\bu, \sigma \Delta_\Gamma \eta \bn)_\Gamma\\
	&\enskip + (\{\gamma_2 \bD(\bu) + (\gamma_3 - \gamma_2) \dv \bu \bI+\gamma_1 \Delta \rho \bI\}\bn, \bv)_\Gamma-(\DV\{\gamma_2 \bD(\bu) + (\gamma_3 - \gamma_2) \dv \bu \bI+\gamma_1 \Delta \rho \bI\}, \bv)_\Omega\\
	&=(\gamma_4 \lambda \bu-\DV\{\gamma_2 \bD(\bu) + (\gamma_3 - \gamma_2) \dv \bu \bI+\gamma_1 \Delta \rho \bI\}, \bv)_\Omega - (\bu, \sigma \Delta_\Gamma \eta \bn)_\Gamma + (\sigma \Delta_\Gamma h\bn, \bv)_\Gamma\\
	&=- (\bu \cdot \bn, \sigma \Delta_\Gamma \eta)_\Gamma + (\sigma \Delta_\Gamma h, \bv \cdot \bn)_\Gamma
	=-\sigma (\lambda h, \Delta_\Gamma \eta)_\Gamma + \sigma(\Delta_\Gamma h, \lambda \eta)_\Gamma\\
	&=0,
\end{align*}
which implies that $\bu=0$.
Since $\lambda \neq 0$, the first and the fifth equations of \eqref{homo} furnish that $\rho = 0$ in $\Omega$ and $h = 0$ on $\Gamma$.

\subsection{Proof of Lemma \ref{lem:Rbound V g}}
To construct $\CV(\lambda)$, let us define
\begin{align*}
	\CD^i_{1j}(\bu^i_j) &= \xi^i_j \gamma^i_{1j} \dv \bu^i_j - \gamma_1 \dv(\xi^i_j \bu^i_j),\\ 
	\CD^i_{2j}(\bu^i_j) &= -\xi^i_j(\gamma^i_{4j})^{-1}\DV(\gamma^i_{2j}\bD(\bu^i_j)) + \gamma_4^{-1} \DV(\gamma_2 \bD(\xi^i_j \bu^i_j)),\\
	\CD^i_{3j}(\bu^i_j) &= -\xi^i_j(\gamma^i_{4j})^{-1}\DV((\gamma^i_{3j} - \gamma^i_{2j}) \dv \bu^i_j \bI) + \gamma_4^{-1} \DV((\gamma_3-\gamma_2)\dv(\xi^i_j \bu^i_j)\bI),\\
	\CD^i_{4j}(\rho^i_j) &= -\xi^i_j(\gamma^i_{4j})^{-1}\DV(\gamma^i_{1j} \Delta \rho^i_j \bI) +  \gamma_4^{-1} \DV(\gamma_1 \Delta(\xi^i_j \rho^i_j)\bI),\\
	\CD_{5j}(\bu^2_j) &= \xi^2_j \gamma^2_{2j}\bD(\bu^2_j)\bn_j -\gamma_2 \bD(\xi^2_j \bu^2_j)\bn,\\
	\CD_{6j}(\bu^2_j) &= \xi^2_j (\gamma^2_{3j} - \gamma^2_{2j}) \dv \bu^2_j \bn_j - (\gamma_3-\gamma_2)\dv(\xi^2_j \bu^2_j)\bn,\\
	\CD_{7j}(\rho^2_j) &= \xi^2_j \gamma^2_{1j} \Delta \rho^2_j \bn_j -\gamma_1 \Delta (\xi^2_j \rho^2_j) \bn,\\
	\CD_{8j}(h_j) &= -\sigma \xi^2_j \Delta_{\Gamma_j} h_j \bn_j + \sigma \Delta_{\Gamma} (\xi^2_j h_j) \bn,\\
	\CD_{9j}(\rho^2_j) &= \xi^2_j \bn_j \cdot \nabla \rho^2_j - \bn \cdot \nabla(\xi^2_j \rho^2_j)
\end{align*}
for $i=1, 2$ and $j \in \N$.
According to the proof of \cite[Lemma 6.5]{S2020}, 
there exist positive constants $C_{M_3}$ and $C_{M_2, M_3}$ independent of $j$, such that
\begin{equation}\label{est:d1}
\begin{aligned}
	\|\nabla \CD^i_{1j} (\bu^i_j)\|_{L_q(\Omega)} &\le C_{M_3} \|(\nabla \bu^i_j. \bu^i_j)\|_{L_q(\Omega \cap B^i_j)},\\
	\|\CD^i_{1j} (\bu^i_j)\|_{L_q(\Omega)} &\le C_{M_3} \|\bu^i_j\|_{L_q(\Omega \cap B^i_j)},\\
	\|\nabla \CD^i_{kj} (\bu^i_j)\|_{L_q(\Omega)} &\le C_{M_3} \|(\nabla \bu^i_j. \bu^i_j)\|_{L_q(\Omega \cap B^i_j)}\quad (k=2, 3),\\
	\|\CD^i_{4j} (\rho^i_j)\|_{L_q(\Omega)} &\le C_{M_3} \|(\nabla^2 \rho^i_j. \nabla \rho^i_j, \rho^i_j)\|_{L_q(\Omega \cap B^i_j)},\\
	\|\nabla^2 \CD_{9j} (\rho^2_j)\|_{L_q(\Omega)} &\le C_{M_2, M_3} \|(\nabla^2 \rho^2_j. \nabla \rho^2_j, \rho^2_j)\|_{L_q(\Omega \cap B^2_j)},\\
	\|\nabla \CD_{9j} (\rho^2_j)\|_{L_q(\Omega)} &\le C_{M_2, M_3} \|(\nabla \rho^2_j, \rho^2_j)\|_{L_q(\Omega \cap B^2_j)},\\
	\|\CD_{9j} (\rho^2_j)\|_{L_q(\Omega)} &\le C_{M_3} \|\rho^2_j\|_{L_q(\Omega \cap B^2_j)}.
\end{aligned}
\end{equation}
Similarly, we consider $\CD_{kj}(\bu^2_j)$, $\CD_{7j}(\rho^2_j)$ and $\CD_{8j}(h_j)$ for $k=5, 6$, and $j \in \N$.
First, we consider $\CD_{kj}(\bu^2_j)$ and $\CD_{7j}(\rho^2_j)$ for $k=5, 6$, and $j \in \N$.
Since $\gamma^2_{\ell j} = \gamma_\ell$ on ${\rm supp~} \xi^2_j$ for $\ell = 1, 2, 3$ by \eqref{local gamma} and $\bn = \bn_j$ on $\Gamma \cap {\rm supp~} \xi^2_j$,
we can rewrite as follows:
\begin{align*}
	\CD_{5j}(\bu^2_j) &= -\gamma_2 (\nabla \xi^2_j \otimes \bu^2_j + \bu^2_j \otimes \nabla \xi^2_j)\bn_j,\\
	\CD_{6j}(\bu^2_j) &= - (\gamma_3-\gamma_2) \nabla \xi^2_j \cdot \bu^2_j \bn_j,\\
	\CD_{7j}(\rho^2_j) &= - \gamma_1\{(\Delta \xi^2_j) \rho^2_j + 2(\nabla \xi^2_j) \cdot (\nabla \rho_j^2)\}\bn_j.
\end{align*}
Thus, by Proposition \ref{general domain} \eqref{g-3} and \eqref{est:normal}, there exist positive constants $C_{M_3}$ and $C_{M_2, M_3}$ independent of $j$, such that
\begin{equation}\label{est:d2}
\begin{aligned}
	\|\nabla \CD_{5j}(\bu^2_j)\|_{L_q(\Omega)} &\le C_{M_2, M_3} \|(\nabla \bu^2_j. \bu^2_j)\|_{L_q(\Omega \cap B^2_j)},\\
	\|\CD_{5j}(\bu^2_j)\|_{L_q(\Omega)} &\le C_{M_3} \|\bu^2_j\|_{L_q(\Omega \cap B^2_j)},\\
	\|\nabla \CD_{6j}(\bu^2_j)\|_{L_q(\Omega)} &\le C_{M_2, M_3} \|(\nabla \bu^2_j, \bu^2_j)\|_{L_q(\Omega \cap B^2_j)},\\
	\|\CD_{6j}(\bu^2_j)\|_{L_q(\Omega)} &\le C_{M_3} \|\bu^2_j\|_{L_q(\Omega \cap B^i_j)},\\
	\|\nabla \CD_{7j}(\rho^2_j)\|_{L_q(\Omega)} &\le C_{M_2, M_3} \|(\nabla^2 \rho^2_j. \nabla \rho^2_j, \rho^2_j)\|_{L_q(\Omega \cap B^2_j)},\\
	\|\CD_{7j}(\rho^2_j)\|_{L_q(\Omega)} &\le C_{M_3} \|(\nabla \rho^2_j, \rho^2_j)\|_{L_q(\Omega \cap B^2_j)}.
\end{aligned}
\end{equation}
Next, we consider $\CD_{8j}(h_j)$ for $j \in \N$.
Since $\Delta_{\Gamma} = \Delta_{\Gamma_j}$ on $\Gamma \cap {\rm supp~} \xi^2_j$, we have
\[
	\CD_{8j}(h_j) = \sigma (\Delta_{\Gamma_j} \xi^2_j) h_j \bn_j + 2\sigma (\nabla_{\Gamma_j} \xi^2_j) \cdot (\nabla_{\Gamma_j} h_j),
\]
where the second term of $\CD_{8j}(h_j)$ is the inner product of $\nabla_{\Gamma_j} \xi^2_j$ and $\nabla_{\Gamma_j} h_j$ on $\Gamma_j$.
Here we set 
\begin{equation}\label{def:delta}
	\widetilde \nabla_{\Gamma_j} f = \Pi_{\Gamma_j} \nabla f, \quad \Pi_{\Gamma_j} = \bI - \bn_j \otimes \bn_j,
	\quad \widetilde \Delta_{\Gamma_j} f = {\rm tr~} \widetilde \nabla_{\Gamma_j}^2 f 
\end{equation}
for the function $f$ defined in a neighborhood of $\Gamma_j$.
Note that $\widetilde \nabla_{\Gamma_j} f = \nabla_{\Gamma_j} f$ on $\Gamma_j$, then $\widetilde \Delta_{\Gamma_j} f = \Delta_{\Gamma_j} f$ on $\Gamma_j$.
Therefore, we have
\[
	\CD_{8j}(h_j) = \sigma (\widetilde \Delta_{\Gamma_j} \xi^2_j) h_j \bn_j + 2\sigma (\widetilde \nabla_{\Gamma_j} \xi^2_j) \cdot (\widetilde \nabla_{\Gamma_j} h_j).
\]
Furthermore, by \eqref{def:delta}
\[
	\widetilde \Delta_{\Gamma_j} f = \Delta f -[{\rm tr~} (\Pi_{\Gamma_j} \nabla \bn_j)] (\bn_j \cdot \nabla)f - \bn_j \cdot[(\nabla^2 f)\bn_j], 
\]
together with Proposition \ref{general domain} \eqref{g-3} and \eqref{est:normal}, implies that 
\begin{equation}\label{est:xi}
	\|\widetilde \nabla_{\Gamma_j} \xi^2_j\|_{H^1_\infty(\R^N)} + \|\widetilde \Delta_{\Gamma_j} \xi^2_j\|_{H^1_\infty(\R^N)} \le C_{M_2, M_3}
\end{equation}
for some positive constant $C_{M_2, M_3}$ independent of $j$.
Therefore, we reach the estimates:
\begin{equation}\label{est:d3}
\begin{aligned}
	\|\nabla \CD_{8j}(h_j)\|_{L_q(\Omega)} &\le C_{M_2, M_3} \|(\nabla^2 h_j. \nabla h_j, h_j)\|_{L_q(\Omega \cap B^2_j)},\\
	\|\CD_{8j}(h_j)\|_{L_q(\Omega)} &\le C_{M_2, M_3} \|(\nabla h_j, h_j)\|_{L_q(\Omega \cap B^2_j)}.
\end{aligned}
\end{equation}

In view of \eqref{remainder},
we set 
\begin{align*}
	\CV^i_{1j}(\lambda)\bH &= \CD^i_{1j}(B^i_j(\lambda)\bH),\\
	\CV^i_{2j}(\lambda)\bH &= \CD^i_{2j}(B^i_j(\lambda)\bH) + \CD^i_{3j}(B^i_j(\lambda)\bH) + \CD^i_{4j}(A^i_j(\lambda)\bH),\\
	\CV_{3j}(\lambda)\bH &= \CD_{5j}(B^2_j(\lambda)\bH) + \CD_{6j}(B^2_j(\lambda)\bH) + \CD_{7j}(A^2_j(\lambda)\bH) + \CD_{8j}(C_j(\lambda)\bH),\\ 
	\CV_{4j}(\lambda)\bH &= \CD_{9j}(A^2_j(\lambda)\bH)
\end{align*}
for $\bH \in \CX^2_q(\Omega)$,
and then we construct $\CV(\lambda)$ by
\begin{align*}
	\CV(\lambda)\bH &= (\CV_1(\lambda)\bH, \CV_2(\lambda)\bH, \CV_3(\lambda)\bH, \CV_4(\lambda)\bH, 0),\\
	\CV_k(\lambda)\bH &= \sum^2_{i=1}\sum^\infty_{j=1} \CV^i_{kj}(\lambda)\bH,\quad
	\CV_\ell(\lambda)\bH = \sum^\infty_{j=1} \CV_{\ell j}(\lambda)\bH
\end{align*}
for $j \in \N$, $k = 1, 2$ and $\ell = 3, 4$.
Since $(\rho^i_j, \bu^i_j, h_j) = (A^i_j(\lambda)\CF^2_\lambda \bF, B^i_j(\lambda)\CF^2_\lambda \bF, C_j(\lambda)\CF^2_\lambda \bF)$, we have
$\CU(\lambda)\bF = \CV(\lambda)\CF^2_\lambda \bF$ for $\bF \in X^0_q(\Omega)$.

Finally, we estimate $\CV(\lambda)$.
Set 
\[
	\CV^i_j(\lambda)\bH =(\CV^i_{1j}(\lambda)\bH, \CV^i_{2j}(\lambda)\bH, \CV_{3j}(\lambda)\bH, \CV_{4j}(\lambda)\bH, 0).
\]
It holds by \eqref{rbdd c}, \eqref{rbdd ab}, \eqref{est:d1}, \eqref{est:d2}, and \eqref{est:d3}, for $n=0, 1$ and for any $m \in \N$, $\{\lambda_\ell\}_{\ell=1}^m \subset \C_{+, \lambda_3}$, and $\{\bH_\ell\}_{\ell=1}^m \subset \CX^2_q(\Omega)$, 
\begin{equation}\label{est:vij}
\begin{aligned}
	&\left(\int^1_0 \left \|\sum^m_{\ell=1} r_\ell(u) \{(\tau \pd_\tau)^n \CF^2_\lambda \CV^i_j(\lambda) \}|_{\lambda = \lambda_\ell} \bH_\ell \right \|_{L_q(\Omega^i_j)}^q \, du\right)^{1/q}\\
	&\le C_{q, \omega_0, R, M_2, M_3} \lambda_3^{-1/2} \left(\int^1_0 \left\|\sum^m_{\ell=1}  r_\ell(u) \bH_\ell \right \|_{\CX^2_q(\Omega \cap B^i_j)}^q\, du \right)^{1/q}.
\end{aligned}
\end{equation}
In the same way that we proved \eqref{rbdd c} by \eqref{rbdd c1}, 
we can obtain \eqref{Rbound V g} by \eqref{est:vij}.
This completes the proof of Lemma \ref{lem:Rbound V g}.

	\section{Proof of Theorem \ref{thm:main Rbound}}\label{sec:general2}
In this section, we shall prove Theorem \ref{thm:main Rbound} by Theorem \ref{thm:main general}.
Since the uniqueness of the solutions has been proved in Theorem \ref{thm:main general}, we only construct the $\CR$-bounded solution operator families for \eqref{main prob}.
Following the method in \cite[Sec.7]{S2020}, we first consider the whole space problem.
Let $E$ be an extension operator from $H^1_q(\Omega)$ to $H^1_q(\R^N)$ and $E_0 \bff$ be the zero extension of $\bff \in L_q(\Omega)^N$.
Let $\gamma_j = \gamma_j(x)$ $(j=1, 2, 3, 4)$ be the coefficients defined on $\overline{\Omega}$ satisfying Assumption \ref{assumption gamma}.
According to \cite[Lemma 7.2]{S2020}, we can extend $\gamma_j$ to uniformly Lipschitz functions $\tgamma_j$ defined on $\R^N$
such that the following conditions hold.
There is a constant $M>0$ depending solely on $\gamma^*$ and $L$ such that 

\begin{enumerate}
\item
	$\tgamma_j(x) = \gamma_j(x)$ for any $x \in \overline{\Omega}$,
\item
	$|\tgamma_j(x) - \tgamma_j(y)| \le (M+(\gamma_*/2)) |x-y|$ for $x, y \in \R^N$,
\item
	$\gamma_*/2 \le \tgamma_j(x) \le M+(\gamma_*/2)$ for any $x \in \R^N$
\end{enumerate}
for $j=1, 2, 3, 4$.
Then we know that the existence of $\CR$-bonded solution operator families for the following problem in $\R^N$ (cf. \cite[Theorem 7.1]{S2020}). 
\begin{equation}\label{extend RN}
\left\{
\begin{aligned}
	&\lambda R + \tgamma_1 \dv \bU = Ed & \quad&\text{in $\R^N$}, \\
	&\lambda \bU - (\tgamma_4)^{-1} 
	\DV\{\tgamma_2 \bD(\bU) + (\tgamma_3 - \tgamma_2) \dv \bU \bI + \tgamma_1  \Delta R \bI\}= E_0 \bff& \quad&\text{in $\R^N$}.
\end{aligned}
\right.
\end{equation}

\begin{prop}\label{prop:extend RN}
Let $1<q<\infty$, and let $\Omega$ be a uniform $C^3$-domain. 
Then there exists a constant $\lambda_4 \ge 1$ such that the following assertions hold true:

$\thetag1$
For any $\lambda \in \C_{+, \lambda_4}$ there exist operators  
\begin{align*}
	&\Phi (\lambda) \in 
{\rm Hol} (\C_{+, \lambda_4}, 
\CL(H^1_q(\R^N) \times L_q(\R^N)^N, H^3_q(\R^N)))\\
	&\Psi (\lambda) \in 
{\rm Hol} (\C_{+, \lambda_4}, 
\CL(H^1_q(\R^N) \times L_q(\R^N)^N, H^2_q(\R^N)^N))
\end{align*}
such that 
for any $(d, \bff) \in H^1_q(\Omega) \times L_q(\Omega)^N$,
\begin{equation*}
	R = \Phi (\lambda) (Ed, E_0\bff), \quad
	\bU = \CB (\lambda) (Ed, E_0\bff)
\end{equation*}
are solutions of problem \eqref{extend RN}.

$\thetag2$
There exists a positive constant $r$ such that
\begin{equation*} 
\begin{aligned}
	&\CR_{\CL(H^1_q(\R^N) \times L_q(\R^N)^N, \fA_q (\R^N))}
(\{(\tau \pd_\tau)^n \CR_\lambda \Phi (\lambda) \mid 
\lambda \in \C_{+, \lambda_4}\}) 
\leq r,\\
	&\CR_{\CL(H^1_q(\R^N) \times L_q(\R^N)^N, \fB_q(\R^N))}
(\{(\tau \pd_\tau)^n \CS_\lambda \Psi (\lambda) \mid 
\lambda \in \C_{+, \lambda_4}\}) 
\leq r
\end{aligned}
\end{equation*}
for $n = 0, 1$, where $\fA_q$ and $\CR_\lambda$ are defined in \eqref{def:space}.
Here, above constant $r$ depend solely on $N$, $q$, $L$, $\gamma_*$, and $\gamma^*$. 

\end{prop}

Setting $\rho = R + \theta$ in \eqref{main prob}, then $(\theta, \bu, h)$ satisfies
\begin{equation}\label{main prob'}
\left\{
\begin{aligned}
	&\lambda \theta + \gamma_1 \dv \bu = \widetilde d & \quad&\text{in $\Omega$}, \\
	&\lambda \bu - \gamma_4^{-1} 
	\DV\{\gamma_2 \bD(\bu) + (\gamma_3 - \gamma_2) \dv \bu \bI + \gamma_1  \Delta \theta \bI\}= \widetilde \bff& \quad&\text{in $\Omega$},\\
	&\{\gamma_2 \bD(\bu) + (\gamma_3 - \gamma_2) \dv \bu \bI + \gamma_1  \Delta \theta \bI \} \bn - \sigma \Delta_{\Gamma} h \bn = \widetilde \bg & \quad&\text{on $\Gamma$},\\
	&\bn \cdot \nabla \theta = \widetilde k & \quad&\text{on $\Gamma$},\\
	&\lambda h - \bu \cdot \bn = \zeta & \quad&\text{on $\Gamma$},
\end{aligned}
\right.
\end{equation}
where
\begin{align*}
	\widetilde d &= \gamma_1 \dv \bU, & \widetilde \bff &= \lambda \bU - \gamma_4^{-1} \DV(\gamma_2 \bD(\bU) + (\gamma_3 - \gamma_2) \dv \bU \bI),\\
	\widetilde \bg &= \bg- \gamma_1 \Delta R \bn, & \widetilde k &= k- \bn \cdot \nabla R. 
\end{align*}
Note that the unit outer normal $\bn$ to $\Gamma$ can be extended to $C^2$ function $\widetilde \bn$ on $\R^N$ satisfying $\|\widetilde \bn\|_{H^2_\infty (\R^N)} \le C$ (cf. \cite[Corollary A.3]{ScS}).
Theorem \ref{thm:main general} and Proposition \ref{prop:extend RN} yield that the solution $(\rho, \bu, h)$ of \eqref{main prob} can be expressed by
\begin{equation}\label{sol of main prob}
\begin{aligned}
	\rho &= R + \theta = \Phi(\lambda)(Ed, E_0 \bff) + \CA_4(\lambda) \CF_\lambda^2(\widetilde d, \widetilde \bff, \widetilde \bg, \widetilde k, \zeta),\\
	\bu &= \CB_4(\lambda) \CF_\lambda^2(\widetilde d, \widetilde \bff, \widetilde \bg, \widetilde k, \zeta),\quad
	h = \CC_4(\lambda) \CF_\lambda^2(\widetilde d, \widetilde \bff, \widetilde \bg, \widetilde k, \zeta).
\end{aligned}
\end{equation}
Let $\bG = (G_1, \dots, G_8) \in \CX^0_q(\Omega)$. Here
$G_1, \dots, G_8$ are corresponding to $d$, $\bff$, $\nabla \bg$, $\lambda^{1/2} \bg$, $\nabla^2 k$, $\lambda^{1/2} \nabla k$, $\lambda k$, and $\zeta$, respectively.
Set
\begin{align*}
	\CD_5(\lambda) \bG = &\CD_4(\lambda) (\gamma_1 \dv \Psi(\lambda)(E G_1, E_0 G_2), \\
	& \lambda \Psi(\lambda)(E G_1, E_0 G_2) - \gamma_4^{-1} \DV(\gamma_2 \bD(\Psi(\lambda)(E G_1, E_0 G_2)) + (\gamma_3 - \gamma_2) \dv \Psi(\lambda)(E G_1, E_0 G_2)\bI),\\
	& G_3 - \nabla (\gamma_1 \Delta \Phi(\lambda)(E G_1, E_0 G_2) \widetilde \bn),
	\enskip G_4 - \lambda^{1/2} (\gamma_1 \Delta \Phi(\lambda)(E G_1, E_0 G_2) \widetilde \bn),\\
	& G_5 - \nabla^2 (\widetilde \bn \cdot \nabla \Phi(\lambda)(E G_1, E_0 G_2)),
	\enskip G_6 - \lambda^{1/2} \nabla (\widetilde \bn \cdot \nabla \Phi(\lambda)(E G_1, E_0 G_2)),\\
	&G_7 - \lambda (\widetilde \bn \cdot \nabla \Phi(\lambda)(E G_1, E_0 G_2)), \enskip \zeta),
\end{align*}
where $\CD \in \{\CA, \CB, \CC\}$.
Then set
\[
	\CA(\lambda) \bG = \Phi(\lambda) (E G_1, E_0 G_2) + \CA_5(\lambda) \bG, \enskip
	\CB(\lambda) \bG = \CB_5(\lambda) \bG, \enskip \CC(\lambda) \bG = \CC_5(\lambda) \bG.
\]
In view of \eqref{sol of main prob}, the solution $(\rho, \bu, h)$ of \eqref{main prob} is written as 
\[
	(\rho, \bu, h) = (\CA(\lambda) \CF_\lambda \bF, \enskip \CB(\lambda) \CF_\lambda \bF,  \enskip \CC(\lambda) \CF_\lambda \bF)
\]
for $\bF = (d, \bff, \bg, k, \zeta)$.
The desired estimate \eqref{rbdd general2} follows from the definition of $\CR$-boundedness (cf. Definition \ref{dfn2}) together with Lemma \ref{lem:5.3}, Theorem \ref{thm:main general}, and Proposition \ref{prop:extend RN},
which completes the proof of Theorem \ref{thm:main Rbound}.

\end{document}